\documentclass[11pt,a4paper]{amsart}
\usepackage{epsfig}
\usepackage{amsmath}
\usepackage{amsthm}
\usepackage{amssymb}
\usepackage{bm}
\usepackage{tikz}
\usepackage{PTSansNarrow}
\usepackage[T1]{fontenc}
\usetikzlibrary{matrix}
\usepackage{graphicx}
\newtheorem{theorem}{Theorem}[section]
\newtheorem{lemma}[theorem]{Lemma}

 \numberwithin{dummy}{section}
\newtheorem{algorithm}{Algorithm}
\newtheorem*{algorithm-PF}{Primal Formulation}
\newtheorem*{algorithm-PMF}{Primal-Mixed Formulation}
\newtheorem*{algorithm-DMF}{Dual-Mixed Formulation}
\newtheorem*{algorithm-PF-C}{Conforming Finite Element Method}
\newtheorem*{algorithm-PMF-C}{Conforming Primal-Mixed Finite Element Method}
\newtheorem*{algorithm-DMF-C}{Conforming Mixed Finite Element Method}
\newtheorem*{algorithm-hmfem}{Hybridized Mixed Finite Element Method}
\newtheorem*{algorithm-primal-wgfem}{Primal WG-FEM}
\newtheorem*{algorithm-primalmixed-wgfem}{Primal-Mixed WG-FEM}
\newtheorem*{algorithm-mixed-wgfem}{Mixed WG-FEM}
\newtheorem*{algorithm-hmwgfem}{Hybridized Mixed WG-FEM}
\newtheorem*{algorithm-hdg}{Hybridizable Discontinuous Galerkin}
\numberwithin{equation}{section}

\newcommand{\bq}{{\bm q}}
\newcommand{\bn}{{\bm n}}
\newcommand{\bx}{{\bm x}}
\newcommand{\bu}{{\bm u}}
\newcommand{\bv}{{\bm v}}
\newcommand{\bV}{{\boldsymbol V}}
\newcommand{\bW}{{\boldsymbol W}}
\newcommand\subsetsim{\mathrel{%
  \ooalign{\raise0.2ex\hbox{$\subset$}\cr\hidewidth\raise-0.8ex\hbox{\scalebox{0.9}{$\sim$}}\hidewidth\cr}}}
\def\leqC{\lesssim}

\def\T{{\mathcal T}}

\def\bn{{\boldsymbol n}}
\def\bq{{\boldsymbol q}}

\def\boldeta{{\boldsymbol\eta}}
\def\bpsi{{\boldsymbol\psi}}

\def\bvarphi{{\boldsymbol\varphi}}
\def\bomega{{\boldsymbol\omega}}

\def\bbQ{\mathbb{Q}}

\def\mathcalQ{{\mathcal Q}}
\def\Wspace{{\mathbb{W}_\varepsilon(\Omega)}}

\def\HnHarmonic{\mathbb{H}_{\varepsilon n,0}(\Omega)}
\def\boldeta{\bm{\eta}}
\newcommand{\pT}{{\partial T}}
\newcommand{\curl}{{\nabla\times}}
\newtheorem{defi}{Definition}[section]
\def\3bar{{|\hspace{-.02in}|\hspace{-.02in}|}}
\renewcommand{\ldots}{\dotsc}
\newcommand{\vertiii}[1]{{\left\vert\kern-0.25ex\left\vert\kern-0.25ex\left\vert #1
\right\vert\kern-0.25ex\right\vert\kern-0.25ex\right\vert}}
\theoremstyle{definition}
\newtheorem{definition}[theorem]{Definition}

\theoremstyle{remark}

\numberwithin{equation}{section}

\usepackage{graphicx}
\usepackage{caption,subcaption}
\makeatletter
\newcommand{\tnorm}{\@ifstar\@tnorms\@tnorm}
\newcommand{\@tnorms}[1]{%
\left|\mkern-2.5mu\left|\mkern-2.5mu\left|
#1
\right|\mkern-2.5mu\right|\mkern-2.5mu\right|
}
\newcommand{\@tnorm}[2][]{%
\mathopen{#1|\mkern-2.5mu#1|\mkern-2.5mu#1|}
#2
\mathclose{#1|\mkern-2.5mu#1|\mkern-2.5mu#1|}
}
\makeatother

\usepackage{hyperref}

\makeatletter
\@namedef{subjclassname@2020}{\textup{2020} Mathematics Subject Classification}
\makeatother

\begin{document}
\title{A new numerical method for div-curl Systems with Low Regularity Assumptions}

\author{Shuhao Cao} \address{Department of Mathematics and Statistics, Washington University in St. Louis, St. Louis, MO 63130} \email{s.cao@wustl.edu} \thanks{The research of Shuhao Cao was partially supported by National Science Foundation Award DMS-1913080.}
\author{Chunmei Wang}
\address{Department of Mathematics \& Statistics, Texas Tech University, Lubbock, TX 79409, USA} \email{chunmei.wang@ttu.edu} \thanks{The research of Chunmei Wang was partially supported by National Science Foundation Award DMS-1849483.}
\author{Junping Wang}
\address{Division of Mathematical Sciences, National Science
Foundation, Alexandria, VA 22314}
\email{jwang@nsf.gov}
\thanks{The research of Junping Wang was supported by the NSF IR/D program, while working at
National Science Foundation. However, any opinion, finding, and
conclusions or recommendations expressed in this material are those
of the author and do not necessarily reflect the views of the
National Science Foundation.}

\begin{abstract} 
This paper presents a numerical method for div-curl systems with normal boundary conditions by using a finite element technique known as primal-dual weak Galerkin (PDWG). The PDWG finite element scheme for the div-curl system has two prominent features in that it offers not only an accurate and reliable numerical solution to the div-curl system under the low $H^\alpha$-regularity ($\alpha>0$) assumption for the true solution, but also an effective approximation of normal harmonic vector fields regardless the topology of the domain. Results of seven numerical experiments are presented to demonstrate the performance of the PDWG algorithm, including one example on the computation of discrete normal harmonic vector fields. 
\end{abstract}



\pagestyle{myheadings}

\maketitle

\section{Introduction}\label{Section:Introduction}
In this paper we are concerned with the development of new numerical methods for div-curl systems equipped with normal boundary conditions. For simplicity,  consider the model problem that seeks a vector field $\bm{u}=\bm{u}(\bx)$ satisfying
\begin{eqnarray}
\nabla\cdot(\varepsilon \bm{u})&=f,\qquad {\rm in}\  \Omega, \label{EQ:div-curl:div-eq}\\
\nabla\times \bm{u}&= \bm{g},\qquad {\rm in}\  \Omega, \label{EQ:div-curl:curl-eq}\\
\varepsilon \bm{u} \cdot \bm{n}&=\phi_1, \qquad {\rm on}\ \Gamma, \label{EQ:div-curl:normalBC}
\end{eqnarray}
where $\Omega\subset {\mathbb R}^3$ is a polyhedral domain that is bounded, open, and connected. $\Gamma=\partial \Omega$ is the boundary of $\Omega$ and is assumed to be the union of a finite number of disjoint surfaces $\Gamma=\bigcup_{i=0}^L\Gamma_i$, where $\Gamma_0$ is the exterior boundary of $\Omega$, $\Gamma_i \ (i=1,  \cdots, L)$ are the other connected components with finite surface areas, and $L$ denotes the number of holes in the domain $\Omega$ geometrically and is known as the second Betti number of $\Omega$ or the dimension of the second de Rham cohomology group of $\Omega$. The Lebesgue-integrable real-valued function $f = f(\bx)$ and the vector field $\bm{g} = \bm{g}(\bx)$ are given in the domain $\Omega$, the coefficient matrix $\varepsilon= \{\varepsilon_{ij}(\bx)\}_{3\times 3}$ is  symmetric and uniformly positive definite in $\Omega$,  and the entries $\varepsilon_{ij}(\bx)$ ($i, j=1,\cdots, 3$) are in $L^{\infty}(\Omega)$. The normal boundary data $\phi_1$ is a given distribution in $H^{-\frac12}(\Gamma)$.

The solution uniqueness for the normal boundary value problem \eqref{EQ:div-curl:div-eq}-\eqref{EQ:div-curl:normalBC} depends on the topology of the domain $\Omega$. It is well-known that the solution uniqueness holds true for simply connected $\Omega$, while the solution is unique up to a normal $\varepsilon$-harmonic function in $\mathbb{H}_{\varepsilon n, 0}(\Omega)$ defined in (\ref{harmo}) when the domain $\Omega$ is not simply connected. The dimension of $\mathbb{H}_{\varepsilon n,0}(\Omega)$ is identical to the first Betti number of $\Omega$, which is 
the rank of the first homology group of $\Omega$ or the dimension of the first de Rham cohomology group of $\Omega$. 

The div-curl system (\ref{EQ:div-curl:div-eq})-(\ref{EQ:div-curl:curl-eq}) arises in many applications such as electromagnetic fields and fluid mechanics. Computational electro-magnetics plays an important role in many areas of science and engineering such as radar, satellite, antenna design, waveguides, optical fibers, medical imaging and design of invisible cloaking devices \cite{r9}. In linear magnetic fields, the function $f(\bx)$ vanishes, $\bu$ represents the magnetic field intensity and $\varepsilon(\bx)$ is the inverse of the magnetic permeability tensor. In fluid mechanics fields, the coefficient matrix $\varepsilon(\bx)$ is diagonal with diagonal entries being the local mass density.  In electrostatics fields, $\varepsilon(\bx)$ is the permittivity matrix. 

Several numerical methods based on finite element approaches have been proposed and analyzed for the div-curl system (\ref{EQ:div-curl:div-eq})-(\ref{EQ:div-curl:curl-eq}). For example, a covolume method was developed in three dimensional space by using the Voronoi-Delaunay mesh pairs \cite{r17}. In \cite{r4}, a finite element method based on different test and trial spaces was analyzed for a div-curl system. In \cite{r11}, the authors introduced a mimetic finite difference scheme for the three dimensional magneto-static problems on general polyhedral partitions. In \cite{r8}, the authors developed a mixed finite element method for three dimensional axisymmetric div-curl systems through a dimension reduction technique based on the cylindrical coordinates in simply connected and axisymmetric domains. In \cite{r3}, the authors developed a least-squares finite element method for two types boundary value problems.  Another least-squares method was proposed for the div-curl problem based on discontinuous elements on nonconvex polyhedral domains in \cite{r2}.  In \cite{r23}, the authors proposed a weak Galerkin finite element method for the div-curl system with either normal or tangential boundary conditions.  Another weak Galerkin scheme was introduced in \cite{li} by using a least-squares approach for the div-curl problem. In \cite{wl}, the authors developed a primal-dual weak Galerkin finite element method for the div-curl system with tangential boundary conditions and proved that the scheme works well with low-regularity assumptions on the exact solution.

Two major challenges for the div-curl model problem (\ref{EQ:div-curl:div-eq})-(\ref{EQ:div-curl:normalBC}) are the low-regularity of the true solution $\bu$ and the difficulties in approximating the normal $\varepsilon$-harmonic vector space $\mathbb{H}_{\varepsilon n, 0}(\Omega)$ due to the topological complexity of the domain $\Omega$. The goal of this paper is to address both challenges through a new numerical method based on the primal-dual weak Galerkin (PDWG) finite element aproach originated in \cite{ww2016} and further developed in \cite{w2018, ww2018, w1, w2, hy1, hy2} for various model problems. Our PDWG numerical method for (\ref{EQ:div-curl:div-eq})-(\ref{EQ:div-curl:normalBC}) has two prominent features over the existing numerical methods: (1)  it offers an effective approximation of the normal $\varepsilon$-harmonic vector space $\mathbb{H}_{\varepsilon n, 0}(\Omega)$ regardless of the topology of the domain $\Omega$; and (2) it provides an accurate and reliable numerical solution for the div-curl system (\ref{EQ:div-curl:div-eq})-(\ref{EQ:div-curl:normalBC}) with low $H^\alpha$-regularity ($\alpha>0$) assumption for the true solution $\bu$. 

The paper is organized as follows. Section \ref{Section:2} is devoted to notations and the derivation of a weak formulation for the div-curl system \eqref{EQ:div-curl:div-eq}-\eqref{EQ:div-curl:normalBC} that involves no partial derivatives over the vector field $\bu$. Section \ref{Section:3} offers a brief review on the discrete weak gradient and discrete weak curl operators. Section \ref{Section:4} is dedicated to the presentation of the PDWG algorithm for the div-curl problem, together with an algorithm for computing discrete normal $\varepsilon$-harmonic vector fields.  
Section \ref{Section:5} is devoted to a discussion of the solution existence and uniqueness for the PDWG scheme. Section \ref{Section:6} contains a convergence theory for the PDWG approximation, and  Section \ref{Section:7} demonstrates the performance of the PDWG algorithm through seven test examples. 

\section{Notations and Preliminaries}\label{Section:2}

We follow the usual notation for Sobolev spaces and norms, see for example \cite{ciarlet, girault-raviart}. For an open bounded domain $D \subset {\mathbb R}^3$ with Lipschitz continuous boundary and any given real number $s\geq 0$, we use $\|\cdot\|_{s,D}$ and $|\cdot|_{s,D}$ to denote the norm and seminorm in the Sobolev space $H^s(D)$, respectively. The space $H^0(D)$ coincides with $L^2(D)$, for which the norm and the inner product are denoted by $\|\cdot\|_D$ and $(\cdot,\cdot)_{D}$, respectively.  We use $H(div_\varepsilon;D)$ to denote the closed subspace of $[L^2(D)]^2$ so that $\nabla\cdot(\varepsilon\bv)\in L^2(D)$. The space $H(div;D)$ corresponds to the case of $\varepsilon=I$. Analogously, we use $H(curl;D)$ to denote the closed subspace of $[L^2(D)]^2$ so that $\nabla\times\bv\in [L^2(D)]^3$. Denote by
$$
H_0(curl;D):=\{\bv\in H(curl; D),\ \bv\times\bn =0 \mbox{ on } \partial D\}
$$
the closed subspace with vanishing tangential boundary values. When $D=\Omega$, we shall drop the script $D$ in the notations. For simplicity, we shall use ``$\lesssim$ '' to denote ``less than or equal to up to a general constant independent of the mesh size or functions appearing in the inequality''.

Introduce the following Sobolev space
\begin{equation}\label{EQ:NewSSpace}
\Wspace=\{\bv\in H_0(curl)\cap H(div_\varepsilon), \ \nabla\cdot(\varepsilon\bv)=0,\ \langle \varepsilon\bv\cdot\bn_i, 1\rangle_{\Gamma_i}=0,\ i=1,\cdots, L\}.
\end{equation}
A vector field $\bv\in [L^2(\Omega)]^3$ is said to be $\varepsilon$-harmonic on $\Omega$ if it is $\varepsilon$-solenoidal and irrotational on $\Omega$. The space of normal $\varepsilon$-harmonic vector fields, denoted by $\mathbb{H}_{\varepsilon n,0}(\Omega)$, consists of all $\varepsilon$-harmonic vector fields satisfying the zero normal boundary condition; i.e.,
\begin{equation}\label{harmo}
\mathbb{H}_{\varepsilon n,0}(\Omega)=\{\bv\in [L^2(\Omega)]^3: \ \curl\bv=0,\
\nabla\cdot(\varepsilon \bv)=0,\ \varepsilon\bv\cdot \bn = 0 \mbox{ on } \Gamma\}.
\end{equation}
When $\varepsilon=I$ is the identity matrix, the space $\mathbb{H}_{\varepsilon n,0}(\Omega)$ shall be denoted as $\mathbb{H}_{n,0}(\Omega)$. Analogously, the space of tangential $\varepsilon$-harmonic vector fields, denoted by $\mathbb{H}_{\varepsilon \tau,0}(\Omega)$, consists of all $\varepsilon$-harmonic vector fields satisfying the zero tangential boundary condition; i.e.,
$$
\mathbb{H}_{\varepsilon \tau,0}(\Omega)=\{\bv\in [L^2(\Omega)]^3: \ \curl\bv=0,\
\nabla\cdot(\varepsilon \bv)=0,\ \bv\times\bn = 0 \mbox{ on } \Gamma\}.
$$

\subsection{A weak formulation}
By testing the equation \eqref{EQ:div-curl:div-eq} against any $\varphi\in H^1(\Omega)$ and then using the normal boundary condition (\ref{EQ:div-curl:normalBC}) we obtain
\begin{equation}\label{EQ:variational-form-1}
(\bm{u}, \varepsilon \nabla \varphi)  = \langle \phi_1, \varphi\rangle - (f,\varphi), \qquad\forall \varphi\in H^1(\Omega).
\end{equation}
Next, we test the equation \eqref{EQ:div-curl:curl-eq} against any $\bm{w}\in H_0(curl;\Omega)$ to obtain
\begin{eqnarray}\label{EQ:variational-form-2}
(\bm{u}, \nabla \times \bm{w}) = (\bm{g},\bm{w}), \qquad \forall \bm{w}\in H_0(curl;\Omega).
\end{eqnarray}
Summing the equations \eqref{EQ:variational-form-1} and \eqref{EQ:variational-form-2} gives the following
\begin{eqnarray*}
(\bm{u}, \varepsilon \nabla \varphi + \nabla \times \bm{w}) = (\bm{g},\bm{w})-(f,\varphi) + \langle \phi_1, \varphi\rangle
\end{eqnarray*}
for all $\varphi \in H^1(\Omega)$ and $\bm{w} \in H_0(curl;\Omega)$.

\begin{defi} A vector-valued function $\bm{u}\in [L^2(\Omega)]^3$ is said to be a weak solution of the normal boundary value problem for the div-curl system (\ref{EQ:div-curl:div-eq})-(\ref{EQ:div-curl:normalBC}) if it satisfies the following equation
\begin{eqnarray}\label{EQ:variational-form}
(\bm{u}, \varepsilon \nabla \varphi + \nabla \times \bpsi ) = (\bm{g},\bpsi)-(f,\varphi) + \langle \phi_1, \varphi\rangle
\end{eqnarray}
for all $\varphi \in H^1(\Omega)$ and $\bpsi \in H_0(curl;\Omega)$.
\end{defi}

The solution to the variational problem \eqref{EQ:variational-form} is generally non-unique. In fact, the homogeneous version of \eqref{EQ:variational-form} seeks $\bu\in [L^2(\Omega)]^3 $ satisfying
\begin{eqnarray}\label{EQ:variational-form-homo}
(\bm{u}, \varepsilon \nabla \varphi + \nabla \times \bpsi) = 0\qquad \forall \varphi\in H^1(\Omega),\ \bpsi\in H_0(curl;\Omega).
\end{eqnarray}
The equation \eqref{EQ:variational-form-homo} is easily satisfied by any $\varepsilon$-harmonic function $\bu=\boldeta\in \mathbb{H}_{\varepsilon n,0}(\Omega)$, and hence the solution non-uniqueness when the $\varepsilon$-harmonic space $\mathbb{H}_{\varepsilon n,0}(\Omega)$ has a positive dimension. The solution to the div-curl system with normal boundary condition is unique when the solution is further required to be $\varepsilon$-weighted $L^2$ orthogonal to $\mathbb{H}_{\varepsilon n,0}(\Omega)$.

\subsection{An extended weak formulation}
Denote by $H_{0c}^1(\Omega)$ the space of functions in $H^1(\Omega)$ with vanishing value on $\Gamma_0$ and constant values on other connected components of the boundary; i.e.,
\begin{equation*}\label{EQ:Nov-11-2014:H0c}
H_{0c}^1(\Omega)=\{\phi\in H^1(\Omega):\ \phi|_{\Gamma_0}=0, \
\phi|_{\Gamma_i}=\alpha_i, \ i=1,\ldots, L\}.
\end{equation*}
Introduce the following bilinear form:
\begin{equation}\label{EQ:div-curl:nbvp:bform}
B(\bu,s; \varphi,\bpsi): = (\bu, \varepsilon \nabla \varphi + \nabla \times \bpsi) + (\bpsi, \varepsilon\nabla s).
\end{equation}
The extended weak formulation for the normal boundary value problem of the div-curl system seeks $(\bu, s)\in [L^2(\Omega)]^3\times H_{0c}^1(\Omega)$ such that
\begin{equation}\label{EQ:weakform-4-nbvp-divcurl}
B(\bu, s; \varphi,\bpsi) = F(\varphi,\bpsi)\qquad \forall \varphi\in H^1(\Omega), \forall \bpsi\in H_0(curl;\Omega),
\end{equation}
where
$$
F(\varphi,\bpsi)= (\bm{g},\bpsi)-(f,\varphi) + \langle \phi_1, \varphi\rangle.
$$

Note that by testing the curl equation in the div-curl system against any $\rho\in H^1_{0c}(\Omega)$ we have
\begin{equation}
\label{EQ:divcurl-compatibility-n:001}
(\bm{g}, \nabla\rho) = 0,\qquad \forall \rho\in H_{0c}^1(\Omega),
\end{equation}
which gives rise to the following compatibility condition:
\begin{equation}
\label{EQ:divcurl-compatibility-n:001new}
\nabla\cdot\bm{g}=0, \quad \langle\bm{g}\cdot\bn_i, 1\rangle_{\Gamma_i} = 0, \
\text{for}\  i=1,\cdots, L.
\end{equation}

\begin{theorem}
Under the compatibility condition \eqref{EQ:divcurl-compatibility-n:001new} for $\bm{g}$, the solution $(\bu; s)$ of \eqref{EQ:weakform-4-nbvp-divcurl} satisfies the following equations:
\begin{eqnarray}
  \nabla\cdot(\varepsilon \bu) &=& f,\qquad \mbox{ in } \Omega, \label{EQ:11:06:201} \\
  \nabla\times\bu &=& \bm{g},\qquad \mbox{ in } \Omega,\label{EQ:11:06:202} \\
  s &=& 0,\qquad \mbox{ in } \Omega,\label{EQ:11:06:2021} \\
\varepsilon\bu\cdot \bn & = & \phi_1,\qquad \mbox{on } \Gamma. \label{EQ:11:06:204}
\end{eqnarray}
\end{theorem}

\begin{proof}
By letting $\bm{\psi}=0$ in \eqref{EQ:weakform-4-nbvp-divcurl} we have
$$
(\bm{u}, \varepsilon \nabla \varphi) = -(f,\varphi) + \langle\phi_1,\varphi\rangle_\Gamma
$$
for all $\varphi\in H^1(\Omega)$. It follows that $\nabla\cdot (\varepsilon\bu) = f$ and
$\varepsilon\bu\cdot\bn = \phi_1 \mbox{ on } \Gamma$,
which leads to \eqref{EQ:11:06:201} and \eqref{EQ:11:06:204}. Next, by letting $\varphi=0$ in \eqref{EQ:weakform-4-nbvp-divcurl} we arrive at
\begin{eqnarray*}
(\bm{u}, \nabla \times \bm{\psi}) + (\bm{\psi}, \varepsilon\nabla s) = (\bm{g},\bm{\psi}), \qquad \forall \bm{\psi}\in H_0(curl;\Omega),
\end{eqnarray*}
which leads to
\begin{eqnarray*}
(\nabla\times\bm{u} + \varepsilon\nabla s, \bm{\psi}) = (\bm{g},\bm{\psi}),
\end{eqnarray*}
and thus
\begin{eqnarray}\label{EQ:11:06:222}
\nabla\times\bm{u} + \varepsilon\nabla s=\bm{g},\qquad \mbox{in} \ \Omega.
\end{eqnarray}
Now from \eqref{EQ:11:06:222} we have
$$
(\nabla\times\bm{u} + \varepsilon\nabla s, \nabla s)  = (\bm{g}, \nabla s),
$$
which, by the usual integration by parts, gives
$$
\langle\bn\times\bm{u}, \nabla s\rangle_\Gamma + (\varepsilon\nabla s, \nabla s)  = (\bm{g}, \nabla s),
$$
and by the boundary condition of $s=const$ on each $\Gamma_i$ and the compatibility condition \eqref{EQ:divcurl-compatibility-n:001new}
$$
(\varepsilon\nabla s, \nabla s) = \langle\bm{u}, \bn\times\nabla s\rangle_\Gamma + (\bm{g}, \nabla s)=0.
$$
It follows that $\nabla s=0$ so that $s\equiv 0$. This completes the proof of the theorem.
\end{proof}

The homogeneous dual problem of \eqref{EQ:weakform-4-nbvp-divcurl} seeks $\lambda \in H^1(\Omega)/\mathbb{R}$ and $\bq \in H_0(curl;\Omega)$ satisfying
\begin{equation}\label{EQ:homo-dual-problem}
B(\bv, r; \lambda, \bq)=0\qquad \forall \bv\in [L^2(\Omega)]^3,\ r\in H_{0c}^1(\Omega).
\end{equation}

\begin{theorem}
The solution to the homogeneous dual problem \eqref{EQ:homo-dual-problem} is unique.
\end{theorem}

\begin{proof}
The problem \eqref{EQ:homo-dual-problem} can be rewritten as
\begin{equation}\label{EQ:div-curl-nbvp:pdwg:001}
(\bm{v}, \varepsilon \nabla \lambda + \nabla \times \bq) + (\bq, \varepsilon\nabla r) = 0
\end{equation}
for all $\bv\in [L^2(\Omega)]^3$ and $r\in H_{0c}^1(\Omega)$. Note that the test against $r\in H^1_{0c}(\Omega)$ and $\bm{v}=0$ ensures $\nabla\cdot(\varepsilon \bq)=0$ and $\langle \varepsilon\bq\cdot \bn, 1\rangle_{\Gamma_i}=0$ for all $i\in \{1,\cdots, L\}$. In addition, by letting $r=0$ and varying $\bm{v}\in [L^2(\Omega)]^3$ we arrive at
$$
\varepsilon \nabla \lambda + \nabla \times \bq = 0,
$$
which, by testing against $\nabla \lambda$, leads to
$$
(\varepsilon\nabla \lambda, \nabla \lambda)=0,
$$
so that $\lambda\equiv 0$ and hence
$$
\nabla\times\bm{q}=0.
$$
Thus, we have
$$
\bm{q}\in \mathbb{H}_{\varepsilon \tau,0}(\Omega),\quad \langle \varepsilon\bm{q}\cdot\bn_i,1\rangle_{\Gamma_i}=0\  \text{for}\ i=1,\cdots, L,
$$
which yields $\bm{q}\equiv 0$.
\end{proof}

\section{Discrete Weak Gradient and Weak Curl}\label{Section:3}
The extended weak formulation \eqref{EQ:weakform-4-nbvp-divcurl} is based on the gradient and curl differential operators. In this section, we shall review the notion of discrete weak gradient and weak curl which forms a corner stone of the weak Galerkin finite element method. To this end, let $T$ be a polyhedral domain with boundary $\partial T$ and unit outward normal direction $\bm{n}$ on $\partial T$. Define the space of weak functions in $T$ by
\begin{equation*}
W(T) =\{v = \{v_0, v_b\} : v_0 \in L_2(T), v_b \in L_2(\partial T)\},
\end{equation*}
where $v_0$ represents the value of $v$ in the interior of $T$, and $v_b$ represents some specific boundary information of $v$. Analogously, define the space of vector-valued weak functions on $T$ by
$$
V(T) =\{\bm{v} = \{\bm{v}_0, \bm{v}_b\} : \bm{v}_0 \in [L_2(T)]^3, \bm{v}_b \in [L_2(\partial T)]^3\}.
$$

Let $P_j(T)$ the space of polynomials on $T$ with total degree $j$ and less. For any $v\in W(T)$, the discrete weak gradient, denoted by $\nabla_{w,j,T} v$, is defined as the unique vector-valued polynomial in $[P_j(T)]^3$ satisfying
\begin{equation}\label{EQ:dis_WeakGradient}
(\nabla_{w,j,T} v, \bm{\varphi})_T = -(v_0,\nabla \cdot \bm{\varphi})_T + \langle v_b , \bm{\varphi}\cdot \bn \rangle_{\partial T},\quad \forall \; \bm{\varphi} \in [P_j(T)]^{3}.
\end{equation}
Analogously, the discrete weak curl of $\bv \in V (T)$, denoted by $\nabla_{w,j,T}\times \bv$, is defined as the unique vector-valued polynomial in $[P_j(T)]^3$, satisfying
\begin{equation}\label{EQ:dis_WeakCurl}
(\nabla_{w,j,T} \times \bv, \bvarphi)_T = (\bv_0,\nabla \times \bvarphi)_T - \langle\bv_b \times \bn, \bvarphi\rangle_{\partial T}, \quad \forall \; \bm{\varphi} \in [P_j(T)]^3.
\end{equation}

\section{PDWG Numerical Algorithm}\label{Section:4}
Let ${\mathcal T}_h$ be a finite element partition of the domain $\Omega$ consisting of
polyhedra that are shape-regular \cite{wy3655}. Denote by ${\mathcal E}_h$ the
set of faces in ${\mathcal T}_h$ and ${\mathcal E}_h^0={\mathcal E}_h \setminus
\partial\Omega$ the set of interior faces. Denote by $h_T$ the diameter of the element $T\in \T_h$ and $h=\max_{T\in {\mathcal T}_h}h_T$ the meshsize of the partition $\T_h$.

For a given integer $k\ge 0$, we introduce the following finite element spaces subordinated to $\T_h$:
\begin{equation*}
\begin{split}
\bV_h=&\{\bm{v}:\ \bm{v}|_T\in [P_k(T)]^3, \forall T\in\T_h \},\\
S_h=&\{ \{s_0,s_b\}: \ s_0|_T\in P_k(T), s_b|_\pT\in P_k(\pT), \forall T\in\T_h, s_b|_{\Gamma_0}=0, s_b|_{\Gamma_i}=c_i\},\\
M_h=&\{ \{\varphi_0,\varphi_b\}: \ \varphi_0|_T\in P_k(T), \varphi_b|_\pT\in P_k(\pT), \forall T\in\T_h,\ \int_\Omega \varphi_0 = 0\},\\
\bW_h=&\{\bm{\psi}=\{\bm{\psi}_0,\bm{\psi}_{b} \}: \ \bm{\psi}_0|_T\in [P_k(T)]^3, \bm{\psi}_{b}|_\pT\in G_k(\pT), \forall T\in\T_h, \bm{\psi}_{b}|_{\Gamma}=0\},
\end{split}
\end{equation*}
where $G_k(\pT):=[P_k(\sigma)]^3\times\bn_\sigma$ is the space of polynomials of degree $k$ in the tangent space of $\pT$.

For simplicity of notation, for $\sigma\in S_h$ or $\sigma\in M_h$, denote by $\nabla_{w}\sigma$ the discrete weak gradient $\nabla _{w, k, T}\sigma$ computed  by using (\ref{EQ:dis_WeakGradient}) on each element $T$; i.e.,
$$
(\nabla _{w} \sigma)|_T=\nabla _{w,k,T}(\sigma|_T), \qquad
\sigma\in S_h \ \text{or}\ \sigma\in M_h.
$$
Analogously, for $\bq\in \bW_h$, denote by $\nabla_{w} \times\bq$ the discrete weak
curl $\nabla_{w, k, T} \times \bq$ computed by using (\ref{EQ:dis_WeakCurl}) on each element $T$; i.e.,
$$
(\nabla _{w} \times \bq)|_T=\nabla _{w, k, T} \times (\bq|_T), \qquad \bq\in \bW_h.
$$

An approximation of the bilinear form $B(\cdot; \cdot)$ is given as follows:
\begin{equation}\label{EQ:Bh}
B_h(\bu_h, s_h; \varphi,\bpsi): = (\bu_h, \varepsilon \nabla_w \varphi+ \nabla_w \times \bpsi) + (\bpsi_0, \varepsilon\nabla_w s_h)
\end{equation}
for $\bu_h\in \bV_h, \ s_h\in S_h,\ \varphi\in M_h,\ \bpsi\in\bW_h$.

The following is the PDWG finite element method for the div-curl model system (\ref{EQ:div-curl:div-eq})-(\ref{EQ:div-curl:normalBC}).

\begin{algorithm}[PDWG Algorithm] For an approximate solution of \eqref{EQ:div-curl:div-eq}-\eqref{EQ:div-curl:normalBC}, one may compute
$\bu_h\in \bV_h$, together with three auxiliary variables $s_h \in S_h$, $\lambda_h\in M_h$, and $\bq_h \in \bW_h$ satisfying
\begin{equation}\label{EQ:PDWG-3d:01}
\left\{
\begin{array}{rl}
s_1(\lambda_h, \bq_h;\varphi,\bpsi) + B_h(\bu_h, s_h; \varphi,\bpsi)&= F(\varphi,\bpsi),\quad \forall  \varphi\in M_h,\ \bpsi \in \bW_h, \\
-s_2(s_h, r)+B_h(\bm{v}, r; \lambda_h, \bq_h) & = 0, \quad \forall \bm{v}\in \bV_h,\ r\in S_h.
\end{array}
\right.
\end{equation}
Here the stabilizer $s_1$ is given by
\begin{eqnarray}
s_1(\lambda_h, \bq_h;\varphi, \bm{\psi})&=&\rho_1 \sum_{T\in \T_h}h_T^{-1} \langle \lambda_0 -\lambda_b, \varphi_0 - \varphi_b\rangle_{\partial T} \\
                                 & & + \rho_2 \sum_{T\in \T_h}h_T^{-1} \langle \bm{q}_0 \times \bm{n} -\bm{q}_{b} \times \bm{n}, \bm{\psi}_0 \times \bm{n} -\bm{\psi}_{b} \times \bm{n}\rangle_{\partial T}, \nonumber
\end{eqnarray}
and $s_2$ is defined accordingly in the space $M_h$ as follows
\begin{eqnarray}
s_2(s_h; r)=\rho_3 \sum_{T\in \T_h}h_T^{-\gamma} \langle s_0 -s_b, r_0 - r_b\rangle_{\partial T},
                                  \nonumber
\end{eqnarray}
where $\gamma\ge -1$ and $\rho_i >0$ are parameters with values at user's discretion.

\end{algorithm}

The PDWG scheme \eqref{EQ:PDWG-3d:01} further provides an approximation of the space of normal $\varepsilon$-harmonic vector fields $\mathbb{H}_{\varepsilon n,0}(\Omega)$, as revealed by Theorem \ref{THM:ErrorEstimate4uh} in that the difference $\boldeta_h=\mathcalQ_h \bu - \bu_h$ is sufficiently close to a true normal $\varepsilon$-harmonic vector field $\boldeta$. For this purpose, we introduce the following notation of discrete normal $\varepsilon$-harmonic functions.

\begin{definition}[discrete normal $\varepsilon$-harmonic functions] A vector field $\boldeta_h\in \bV_h$ is said to be a discrete normal $\varepsilon$-harmonic function if there exists a vector field $\bu\in H(div_\varepsilon; \Omega)\cap H(curl; \Omega)$ such that
\begin{equation}\label{EQ:discrete-harmonic-function}
\boldeta_h=\mathcalQ_h\bu - \bu_h,
\end{equation}
where $\mathcalQ_h$ is the $L^2$ projection operator onto the finite element space $\bV_h$ and $\bu_h$ is the solution of \eqref{EQ:PDWG-3d:01} for a div-curl system \eqref{EQ:div-curl:div-eq}-\eqref{EQ:div-curl:normalBC} with load functions $f$, $\bm{g}$, and $\phi_1$ determined by $\bu$.
\end{definition} 

In practical computation, a discrete normal $\varepsilon$-harmonic function can be readily obtained from \eqref{EQ:discrete-harmonic-function} by choosing a smooth vector field $\bu$ and one solving of the PDWG system \eqref{EQ:PDWG-3d:01}.

\section{Solution Existence and Uniqueness}\label{Section:5}
Introduce two semi-norms as follows:
\begin{equation}\label{norm}
\3bar (\lambda_h, \bq_h) \3bar =\Big(s_1(\lambda_h, \bq_h;\lambda_h, \bq_h)\Big)^{\frac{1}{2}}, \qquad \lambda_h\in M_h,  \ \bq_h \in \bW_h,
\end{equation}
\begin{equation}\label{norm2}
\3bar s_h \3bar =\Big(s_2(s_h; s_h)\Big)^{\frac{1}{2}}, \qquad s_h \in S_h.
\end{equation}
For simplicity, assume that $\varepsilon$ is piecewise constant with respect to the partition $\T_h$. Note that all the results can be generalized to piecewise smooth $\varepsilon$ without any difficulty.

Denote by $Q_0$ the $L^2$ projection operator onto $P_k(T)$ and $Q_b$ the $L^2$ projection operator onto $P_k(\sigma)$ on each face $\sigma \in \partial T$. Denote by $Q_h$ the projection operator onto the weak finite element space $S_h$ or $M_h$ such that
$$
(Q_h w)|_T = \{Q_0 w|_T,Q_b w|_{\pT}\}.
$$
Analogously, we use $\bbQ_0$, $\bbQ_b$ and  $\bbQ_h$ to denote the $L^2$ projection operators onto $[P_k(T)]^3$, $G_k(\sigma)=[P_k(\sigma)]^3\times\bn_\sigma$, and $\bW_h$, respectively. The $L^2$ projection operator onto the finite element space $\bV_h$ is denoted as $\mathcalQ_h$.

\begin{lemma}\label{Lemma5.1} \cite{wy3655} The $L^2$ projections $Q_h$ and ${\mathcal Q}_h$ satisfy the commutative property
 \begin{equation}\label{l}
 \nabla_{w}(Q_h w) = {\mathcalQ}_h(\nabla w), \qquad \forall w\in H^1(T),
 \end{equation}
 \begin{equation}\label{l-2}
 \nabla_{w}\times(\bbQ_h \bpsi) = {\mathcal Q}_h(\nabla \times \bpsi),  \qquad \forall  \bpsi\in H(curl; T).
 \end{equation}
 \end{lemma}

\begin{theorem}
The kernel of the matrix of the PDWG finite element scheme \eqref{EQ:PDWG-3d:01} is given by
$$
K_h = \{ (\bv_h, s_h = 0, \lambda_h = 0, \bq_h = 0):\ \bv_h\in \bV_h\cap \mathbb{H}_{\varepsilon n,0}(\Omega)\}.
$$
In other words, the kernel of the matrix of the PDWG scheme \eqref{EQ:PDWG-3d:01} is isomorphic to the subspace of $\mathbb{H}_{\varepsilon n,0}(\Omega)$ consisting of harmonic functions that are piecewise polynomial of degree $k$.
\end{theorem}

\begin{proof} Let $(\bu_h, s_h, \lambda_h,\bm{q}_h)$ be a solution of \eqref{EQ:PDWG-3d:01} with homogeneous data. It follows that
\begin{eqnarray}
 &&s_1(\lambda_h, \bm{q}_h;\lambda_h, \bm{q}_h) = 0, \quad s_2(s_h,s_h)=0, \label{eq:11:08:100}\\
&& (\bm{u}_h, \varepsilon\nabla_w \varphi + \nabla_w \times\bm{\psi}) +(\bm{\psi}_0,\varepsilon\nabla_w s_h) = 0,\qquad\forall \varphi\in M_h, \bm{\psi}\in\bW_h,\label{eq:11:08:101}\\
&&(\bm{q}_0, \varepsilon\nabla_w r)+(\bm{v}, \varepsilon \nabla_w \lambda_h +\nabla_w \times \bm{q}_h) = 0,\qquad \forall \bm{v}\in\bV_h, r\in S_h.\label{eq:11:08:102}
\end{eqnarray}
From \eqref{eq:11:08:100}  we have
\begin{eqnarray} \label{eq:2020:01:16:001}
\lambda_0=\lambda_b,\;\bm{q}_0\times \bm{n}=\bm{q}_{b}\times \bm{n},\ s_0=s_b, \  \text{ on } \partial T,
\end{eqnarray}
so that $\lambda_0\in C(\Omega)$, $s_0\in C(\Omega)$ and $\bm{q}_0\in H_0(curl;\Omega)$. Hence,
\begin{eqnarray}\label{EQ:Nov-26:01}
\nabla \lambda_0 = \nabla_w \lambda_h ,\; \nabla \times \bm{q}_0 = \nabla_w \times \bm{q}_h.
\end{eqnarray}

Next, by letting $r=0$ and varying $\bm{v}$ in \eqref{eq:11:08:102} we have
$$
\varepsilon \nabla_w \lambda_h + \nabla_w \times \bm{q}_h = 0,
$$
which, together with \eqref{EQ:Nov-26:01}, implies
\begin{equation}\label{EQ:11:10:100}
\varepsilon \nabla \lambda_0 + \nabla\times \bm{q}_0 = 0.
\end{equation}
From $\bm{q}_0\in H_{0}(curl;\Omega)$, we have
\begin{eqnarray*}
(\varepsilon \nabla \lambda_0 + \nabla\times \bm{q}_0, \nabla \lambda_0 )
&=&(\varepsilon \nabla \lambda_0, \nabla \lambda_0) + (\nabla\times \bm{q}_0, \nabla \lambda_0)\\
&=&(\varepsilon \nabla \lambda_0, \nabla \lambda_0) + \langle \bm{n}\times \bm{q}_0, \lambda_0\rangle\\
& = & (\varepsilon \nabla \lambda_0, \nabla \lambda_0).\end{eqnarray*}
Thus,
\begin{eqnarray*}
(\varepsilon \nabla \lambda_0, \nabla \lambda_0 )=0,
\end{eqnarray*}
which gives $\nabla \lambda_0 = \bm{0}$, and hence $\lambda_0 \equiv 0$ as a function with mean value 0.  This further leads to $\lambda_b\equiv 0$. Thus, from \eqref{EQ:11:10:100} we have
$$
\nabla \times \bm{q}_0=0, \qquad \text{in}\ \Omega.
$$
Observe that $\bm{q}_0$ satisfies
$$
(\bm{q}_0, \varepsilon\nabla_w r)=0\qquad \forall r\in S_h,
$$
which leads to $\bm{q}_0 \in H(div_\varepsilon; \Omega)$ and
$$
\nabla\cdot(\varepsilon\bm{q_0})=0,\quad \langle \bm{q}_0\cdot\bn_i,1\rangle_{\Gamma_i}=0,
$$
for $i=1,2,\cdots L$. This, together with $\nabla\times\bm{q}_0=0$ and $\bm{q_0}\in H_0(curl;\Omega)$ shows that $\bm{q}_0 \equiv 0$, and hence $\bm{q}_{b}=\bn\times(\bm{q}_{b}\times\bn)=\bn\times 0=0$.

Next, from the Helmholtz decomposition \eqref{EQ:helmholtz-2}, we have
$$
\bm{u}_h=\varepsilon^{-1} \nabla\times\bm{\tilde\psi}+\nabla\tilde\phi+\bm{\tilde\eta},
$$
with some $\bm{\tilde\eta}\in \mathbb{H}_{\varepsilon n,0}(\Omega)$ and $\bm{\tilde\psi}\in H_0(curl;\Omega)$ satisfying $\nabla\cdot(\varepsilon\bm{\tilde\psi}) =0$ and $\langle\varepsilon\bm{\tilde\psi}\cdot\bn_i,1\rangle_{\Gamma_i}=0$ for $i=1,\cdots, L$.
From $s_2(s_h,s_h)=0$ we have $s_0=s_b$ on $\pT$ for each element $T\in \T_h$ so that $s_0\in H^1(\Omega)$. It follows that $\nabla_w s_h = \nabla s_0$. If $\mathbb{H}_{\varepsilon n,0}(\Omega)$ has dimension 0, then $\bm{\tilde\eta}=0$. In \eqref{eq:11:08:101}, by choosing the test function $\phi$ and $\bm{\psi}$ to be the $L^2$ projection of the corresponding function in the Helmholtz decomposition we arrive at
\begin{equation}
\begin{split}
0=&(\bm{u}_h, \varepsilon\nabla_w Q_h\tilde\varphi + \nabla_w \times \bbQ_h\bm{\tilde\psi})+(\bbQ_0 \bm{\tilde\psi},\varepsilon\nabla_w s_h)\\
=&(\bm{u}_h, {\mathcal Q}_h\varepsilon\nabla\tilde\varphi + {\mathcal Q}_h\nabla \times \bm{\tilde\psi}) +(\bm{\tilde\psi},\varepsilon\nabla s_0)\\
=&(\bm{u}_h, \varepsilon\nabla\tilde\varphi + \nabla \times \bm{\tilde\psi}) +(\bm{\tilde\psi},\varepsilon\nabla s_0)\\
=&(\varepsilon\bm{u}_h, \bm{u}_h - \bm{\tilde\eta}) +(\bm{\tilde\psi},\varepsilon\nabla s_0)\\
=&(\varepsilon(\bm{u}_h-\bm{\tilde\eta}), \bm{u}_h- \bm{\tilde\eta}),
\end{split}
\end{equation}
which leads to $\bm{u}_h-\bm{\tilde\eta}=0$; i.e., $\bm{u}_h$ is a harmonic function. As a harmonic function and piecewise polynomial of degree $k$, the first term on the left-hand side of \eqref{eq:11:08:101} becomes to be zero for all test functions $\varphi\in M_h$ and $\bm{\psi}\in\bW_h$. It follows that $\nabla_w s_h=0$ so that $\nabla s_0= \nabla_w s_h =0$ and hence $s_0\equiv 0$, so is $s_b \equiv 0$.
\end{proof}

The following is our main result concerning the solution existence and uniqueness of the numerical scheme \eqref{EQ:PDWG-3d:01}.

\begin{theorem}
The PDWG finite element scheme \eqref{EQ:PDWG-3d:01} has one and only one solution for all the components except $\bm{u}_h$. The solution $\bm{u}_h$ is unique up to a discrete harmonic function $\bm{\eta}_h\in \mathbb{H}_{\varepsilon n,0}(\Omega)$ that is a piecewise polynomial of degree $k$.
\end{theorem}

For the PDWG element of lowest order (i.e., $k=0$), a discrete harmonic function $\bm{\eta}_h\in \mathbb{H}_{\varepsilon n,0}(\Omega)$ would be a piecewise constant vector field that is continuous across each interior element interface and has vanishing value on the domain boundary along the normal direction. It follows that one must have $\bm{\eta}_h\equiv 0$, or equivalently, the PDWG finite element scheme \eqref{EQ:PDWG-3d:01} has one and only one solution for all the components. 

\section{Error Analysis}\label{Section:6}
For the exact solution $\{\bu;s=0\}$ of the div-curl system, we have from (\ref{EQ:dis_WeakGradient}) and (\ref{EQ:dis_WeakCurl}) that
\begin{equation}\label{EQ:div-curl:EE:November-08:100}
\begin{split}
B_h(\mathcalQ_h\bu, Q_h s; \varphi,\bpsi) =& (\mathcalQ_h\bu, \varepsilon \nabla_w \varphi + \nabla_w \times \bpsi) + (\bpsi_0, \varepsilon\nabla_w Q_h s)\\
=& (\mathcalQ_h\bu, \varepsilon \nabla \varphi_0 + \nabla \times \bpsi_0) \\
& + \langle \mathcalQ_h\bu, \varepsilon\bn(\varphi_b-\varphi_0) + (\bpsi_0-\bpsi_b)\times\bn\rangle_{\partial\T_h}\\
=& (\bu, \varepsilon \nabla \varphi_0 + \nabla \times \bpsi_0) \\
& + \langle \mathcalQ_h\bu, \varepsilon\bn(\varphi_b-\varphi_0) + (\bpsi_0-\bpsi_b)\times\bn\rangle_{\partial\T_h}\\
=& -(\nabla\cdot(\varepsilon\bu),\varphi_0) + (\nabla\times\bu, \bpsi_0) \\
& + \langle \bu, \varepsilon\bn(\varphi_0-\varphi_b) + (\bpsi_b-\bpsi_0)\times\bn\rangle_{\partial\T_h}\\
& + \langle \mathcalQ_h\bu, \varepsilon\bn(\varphi_b-\varphi_0) + (\bpsi_0-\bpsi_b)\times\bn\rangle_{\partial\T_h}\\
& + \langle {\phi}_1, \varphi_b\rangle_{\partial\Omega}\\
=& \langle {\phi}_1, {\varphi}_b\rangle_{\partial\Omega} - (f,\varphi_0) +(\bm{g},\bm{\psi}_0)\\
&+ \langle \bu-\mathcalQ_h\bu, \varepsilon\bn(\varphi_0-\varphi_b) + (\bpsi_b-\bpsi_0)\times\bn\rangle_{\partial\T_h}.
\end{split}
\end{equation}
Combining the above equation with the fact that $\lambda=0, \ \bm{q}=0$ we obtain
\begin{equation}\label{EQ:div-curl:EE:November-08:ee:01}
\begin{split}
&s_1(Q_h\lambda-\lambda_h, \bbQ_h\bm{q}-\bm{q}_h;\varphi, \bm{\psi})
+ B(\mathcalQ_h\bu-\bu_h, Q_h s - s_h; \varphi,\bpsi) \\
= & \langle \bu-\mathcalQ_h\bu, \varepsilon\bn(\varphi_0-\varphi_b) + (\bpsi_b-\bpsi_0)\times\bn\rangle_{\partial\T_h}.
\end{split}
\end{equation}

The second error equation can be easily obtained as follows:
\begin{equation}\label{EQ:div-curl:EE:November-08:ee:02}
-s_2(Q_hs-s_h, r)
+ B(\bm{v}, r; Q_h\lambda - \lambda_h, \bbQ_h\bm{q}- \bm{q}_h) = 0,
\end{equation}
where we have used the fact that $s=0, \ \bm{q}=0$, and $\lambda=0$.

Denote the error functions by
$$
e_{\bu}= \mathcalQ_h\bu - \bu_h, \ e_s=Q_h s - s_h,\ e_\lambda=Q_h\lambda - \lambda_h, \ e_{\bm{q}}=\bbQ_h\bm{q}-\bm{q}_h.
$$

\begin{theorem}\label{THM:ErrorEstimate4Triple} For the numerical solution $\bu_h, \ s_h, \ \lambda_h,\ \bm{q}_h$ arising from the PDWG scheme \eqref{EQ:PDWG-3d:01}, the following estimate holds true:
\begin{equation}\label{EQ:2020-02-09:100}
\begin{split}
\3bar (e_\lambda, e_{\bm{q}})\3bar + \3bar e_s \3bar \lesssim h^{k+\theta}\|\bu\|_{k+\theta},
\end{split}
\end{equation}
provided that $\bu\in [H^{k+\theta}(\Omega)]^3$ for $\theta\in (1/2,1]$.
\end{theorem}

\begin{proof}
From \eqref{EQ:div-curl:EE:November-08:ee:01} and \eqref{EQ:div-curl:EE:November-08:ee:02} we have
\begin{equation*}
\begin{split}
s_1(e_\lambda, e_{\bm{q}}; e_\lambda, e_{\bm{q}}) + s_2(e_s, e_s)
= &\langle \bu-\mathcalQ_h\bu, \varepsilon\bn(e_{\lambda,0}-e_{\lambda,b}) + (e_{\bm{q},b}-e_{\bm{q},0})\times\bn\rangle_{\partial\T_h},
\end{split}
\end{equation*}
which leads to
\begin{equation}\label{EQ:error-estimate-n-part01}
\3bar (e_\lambda, e_{\bm{q}})\3bar^2 + \3bar e_s \3bar^2 \lesssim h^{k+\theta}\|\bu\|_{k+\theta}
\3bar ( e_\lambda, e_{\bm{q}}) \3bar,
\end{equation}
where $\theta\in (1/2,1]$ and $k$ is the order of polynomials for the finite element space $\bV_h$.
This completes the proof of the theorem.
\end{proof}

To derive an estimate for the error function $e_\bu$, we use the Helmholtz decomposition \eqref{EQ:helmholtz-2} to obtain a $\tilde\phi\in H^1(\Omega)$, $\bm{\tilde\psi}\in {H}_0(curl;\Omega)$, and $\tilde\boldeta\in \HnHarmonic$ such that
\begin{equation}\label{EQ:2020-02-09:101}
e_\bu = \varepsilon^{-1} \nabla\times\bm{\tilde\psi} + \nabla\tilde\phi + \tilde\boldeta.
\end{equation}
Assume the following $H^\alpha$-regularity holds true for some fixed $\alpha\in (1/2, 1]$:
\begin{equation}\label{EQ:regularity-assumption-helmholtz-01new}
\|\bm{\tilde\psi}\|_\alpha + \|\tilde\phi\|_\alpha \lesssim \|e_\bu - \tilde\boldeta\|_0.
\end{equation}
The following is the main convergence result of this paper.

\begin{theorem}\label{THM:ErrorEstimate4uh} Let $\bu$ be a solution of the div-curl system (\ref{EQ:div-curl:div-eq})-(\ref{EQ:div-curl:normalBC}). Assume that the Helmholtz decomposition \eqref{EQ:2020-02-09:101} has the $H^\alpha$-regularity estimate \eqref{EQ:regularity-assumption-helmholtz-01new}. For a numerical solution $\bu_h, \ s_h, \ \lambda_h,\ \bm{q}_h$ arising from \eqref{EQ:PDWG-3d:01}, there exists a harmonic function $\boldeta\in \HnHarmonic$ such that the following estimate holds true:
\begin{equation}\label{EQ:2020-02-09:102}
\begin{split}
\|\varepsilon^{1/2}(\bu_h+\boldeta-\mathcalQ_h\bu)\| \lesssim  h^{k+\theta+\alpha-1} \|\bu\|_{k+\theta},
\end{split}
\end{equation}
provided that $\bu\in [H^{k+\theta}(\Omega)]^3$ for $\theta\in (1/2,1]$.
\end{theorem}

\begin{proof}
From the first error equation \eqref{EQ:div-curl:EE:November-08:ee:01}, we have
\begin{equation}\label{EQ:div-curl:EE:November-08:ee:01-new}
\begin{split}
s_1(e_\lambda, e_{\bm{q}};\varphi, \bm{\psi})
+ B(e_\bu, e_s; \varphi,\bpsi)
= &\langle \bu-\mathcalQ_h\bu, \varepsilon\bn(\varphi_0-\varphi_b) + (\bpsi_b-\bpsi_0)\times\bn\rangle_{\partial\T_h}.
\end{split}
\end{equation}
By letting $\varphi=Q_h\tilde\phi$ and $\bm{\psi}=Q_h\bm{\tilde\psi}$, we obtain from Lemma \ref{Lemma5.1}
\begin{equation}\label{EQ:11:08:300}
\begin{split}
B(e_\bu, e_s; \varphi,\bm{\psi})=&(e_\bu, \varepsilon\nabla_w Q_h\tilde\phi + \nabla_w\times \bbQ_h\bm{\tilde\psi})+(\bbQ_0\bm{\tilde\psi},\varepsilon\nabla_w e_s)\\
=&(e_\bu, \varepsilon\mathcalQ_h\nabla_w \tilde\phi + \mathcalQ_h\nabla_w\times \bm{\tilde\psi})+(\bbQ_0\bm{\tilde\psi},\varepsilon\nabla_w e_s)\\
= & (e_\bu, \varepsilon\nabla \tilde\phi + \nabla\times \bm{\tilde\psi})+(\bbQ_0\bm{\tilde\psi},\varepsilon\nabla_w e_s)\\
= & (\varepsilon e_\bu, e_\bu - \tilde\boldeta ) + (\varepsilon \bbQ_0\bm{\tilde\psi},\nabla_w e_s)\\
= & (\varepsilon (e_\bu - \tilde\boldeta), e_\bu - \tilde\boldeta ) + (\varepsilon \bbQ_0\bm{\tilde\psi},\nabla_w e_s).
\end{split}
\end{equation}
From the definition of the weak gradient we have
\begin{equation*}
\begin{split}
(\varepsilon \bbQ_0\bm{\tilde\psi},\nabla_w e_s) = &(\varepsilon \bbQ_0\bm{\tilde\psi},\nabla e_{s,0}) + \langle \varepsilon \bbQ_0\bm{\tilde\psi}\cdot\bn, e_{s,b}-e_{s,0}\rangle_{\partial\T_h} \\
= & (\varepsilon\bm{\tilde\psi},\nabla e_{s,0}) + \langle\varepsilon \bbQ_0\bm{\tilde\psi}\cdot\bn, e_{s,b}-e_{s,0}\rangle_{\partial\T_h} \\
= & -(\nabla\cdot(\varepsilon\bm{\tilde\psi}),\nabla e_{s,0}) +\langle \varepsilon\bm{\tilde\psi}\cdot\bn, e_{s,0} \rangle_{\partial\T_h} + \langle \varepsilon \bbQ_0\bm{\tilde\psi}\cdot\bn, e_{s,b}-e_{s,0}\rangle_{\partial\T_h}\\
= & \langle \varepsilon\bm{\tilde\psi}\cdot\bn, e_{s,0} - e_{s,b}\rangle_{\partial\T_h} + \langle\varepsilon \bbQ_0\bm{\tilde\psi}\cdot\bn, e_{s,b}-e_{s,0}\rangle_{\partial\T_h}\\
= & \langle \varepsilon(\bm{\tilde\psi} - \bbQ_0\bm{\tilde\psi})\cdot\bn, e_{s,0} - e_{s,b}\rangle_{\partial\T_h}.
\end{split}
\end{equation*}
Substituting the above into \eqref{EQ:11:08:300} then \eqref{EQ:div-curl:EE:November-08:ee:01-new} yields
\begin{equation*}
\begin{split}
\|\varepsilon^{1/2}(e_\bu-\tilde\boldeta)\|^2 = & B(e_\bu, e_s; \varphi,\bm{\psi}) - \langle \varepsilon(\bm{\tilde\psi} - \bbQ_0\bm{\tilde\psi})\cdot\bn, e_{s,0} - e_{s,b}\rangle_{\partial\T_h}\\
= & \langle \bu-\mathcalQ_h\bu, \varepsilon\bn(\varphi_0-\varphi_b) + (\bpsi_b-\bpsi_0)\times\bn\rangle_{\partial\T_h} - s_1(e_\lambda, e_{\bm{q}};\varphi, \bm{\psi}) \\
& - \langle\varepsilon (\bm{\tilde\psi} - \bbQ_0\bm{\tilde\psi})\cdot\bn, e_{s,0} - e_{s,b}\rangle_{\partial\T_h},
\end{split}
\end{equation*}
which leads to
\begin{equation*}
\begin{split}
\|\varepsilon^{1/2}(e_\bu-\tilde\boldeta)\|^2 \lesssim
& h^\theta \|\bu-\mathcalQ_h\bu\|_\theta \3bar (\varphi, \bpsi)\3bar + \3bar (e_\lambda, e_{\bm{q}})\3bar
\3bar (\varphi, \bm{\psi})\3bar \\
& + h^{\alpha +(\gamma-1)/2} \|\bm{\tilde\psi}\|_\alpha  \3bar e_{s}\3bar.
\end{split}
\end{equation*}
It can be proved that
$$
\3bar (\varphi, \bpsi)\3bar \lesssim h^{\alpha-1} \|(\tilde\varphi, \tilde\bpsi)\|_\alpha.
$$
It follows that
\begin{equation*}
\begin{split}
\|\varepsilon^{1/2}(e_\bu-\tilde\boldeta)\|^2 \lesssim
& h^\theta \|\bu-Q_h\bu\|_\theta  h^{\alpha-1} \|(\tilde\varphi, \tilde\bpsi)\|_\alpha + \3bar (e_\lambda, e_{\bm{q}})\3bar  h^{\alpha-1} \|(\tilde\varphi, \tilde\bpsi)\|_\alpha \\
& +  h^{\alpha+(\gamma-1)/2} \|\bm{\tilde\psi}\|_\alpha  \3bar e_{s}\3bar.
\end{split}
\end{equation*}
so that
\begin{equation*}
\begin{split}
\|\varepsilon^{1/2}(e_\bu-\tilde\boldeta)\| \lesssim
& h^{\alpha+\theta-1} \|\bu-Q_h\bu\|_\theta  + h^{\alpha-1} \3bar (e_\lambda, e_{\bm{q}}) \3bar
+ h^{\alpha+(\gamma-1)/2} \3bar e_{s}\3bar\\
\leq & h^{k+\theta+\alpha-1} \|\bu\|_{k+\theta},
\end{split}
\end{equation*}
which gives rise to the error estimate \eqref{EQ:2020-02-09:102}.
\end{proof}

\section{Numerical Experiments}\label{Section:7}
In this section, we present some numerical results for the PDWG finite
element method proposed and analyzed in previous sections. For simplicity, we choose the lowest order PDWG element; i.e., $k = 0$ so that the solution $\boldsymbol{u}$ is approximated by discontinuous piecewise constant vector fields. The exact solution $\boldsymbol{u}$ has various regularities ranging from smooth to corner singular. The computational domain
includes convex and non-convex polyhedral regions; some have cavities or multiple toroidal topology. The implementation uses an open-source and publicly available MATLAB package iFEM \cite{chen2008ifem}. The computational domain is first partitioned into cubes, and each cube is further divided into 6 tetrahedra of equi-volume to form a shape-regular finite element partition. On each tetrahedral element $T$ with boundary $\partial T = \cup_{i=1}^{4} F_i$, the local finite element space consists of functions given as follows:
\begin{align}
& \boldsymbol{u}_{h}|_{T} \in \left[P_{0}(T)\right]^{3},
\\
& s_{h}|_{T}=\left\{s_{0}, s_{b}\right\}
\in\left\{P_{0}(T), \Pi_{i=1}^4 P_{0}(F_i)\right\},
\\
& \varphi_{h}|_{T}=\left\{\varphi_{0}, \varphi_{b}\right\}
\in\left\{P_{0}(T), \Pi_{i=1}^4 P_{0}(F_i)\right\},
\\
& \boldsymbol{\psi}_{h}|_{T}=\left\{\boldsymbol{\psi}_{0},
\boldsymbol{\psi}_{b}\right\}
\in\left\{\left[P_{0}(T)\right]^{3}, T_0(\partial T)\right\}.
\end{align}
Among the spaces above, $T_0(\partial T)$ is tangetial to the boundary and is given by
\[
T_0(\partial T):= \{\boldsymbol{\psi}: \boldsymbol{\psi}_{F_{i,j}} \in
\left[P_{0}(F_{i})\right]^{3}\times \boldsymbol{n}_{F_i}, \; j=1,2 \text{ and }
i=1,2,3,4 \},
\]
where $\boldsymbol{n}_{F_i}$ is the outer unit normal vector to face $F_i$. The
basis functions for the first three spaces are straightforward. For
$T_0(\partial T)$, on each face $F_i$ we choose the normalized vectors representing
the directional vector of any 2 edges ($j=1,2$) among 3 on $\partial F_i$, such that
its weak curl is the co-normal vector of this edge with respect to $F_i$ ($i=
1,2,3,4$),
\[
\nabla_w\times \left\{\boldsymbol{0}, \boldsymbol{\psi}_{F_{i,j}}\right\}
= 3\boldsymbol{\psi}_{F_{i,j}} \times (\nabla \zeta_i),
\]
where $\zeta_i$ is the barycentric coordinate associated with the vertex opposite to
face $F_i$.

For each test problem, we specify a vector field $\boldsymbol{u}$ as the true solution,
while the right-hand side of the div-curl system \eqref{EQ:div-curl:div-eq}-\eqref{EQ:div-curl:normalBC} are computed accordingly. We shall
evaluate the following errors for the PDWG finite element solution:
\begin{align}
& \|\varepsilon^{1/2} \mathbf{e}_{\boldsymbol{u}}\| :=
\|\varepsilon^{1/2}(\boldsymbol{u} -
\boldsymbol{u}_h)\|,
\\
& \tnorm{(e_{\lambda}, \mathbf{e}_{\boldsymbol{q}})} : =
\big(s_1(\lambda_h, \boldsymbol{q}_h; \lambda_h, \boldsymbol{q}_h) \big)^{1/2},
\\
& \tnorm{e_{s}} : = \big(s_2(s_h; s_h) \big)^{1/2},
\end{align}
where the $L^2(\Omega)$-norm $\|\varepsilon^{1/2} \mathbf{e}_{\boldsymbol{u}}\|$ is computed by using a higher order Gaussian quadrature on each element.

\subsection{Example 1}
\label{subsec:ex1}
$\Omega = (0,1)^3$, $\varepsilon = \operatorname{diag}(3,2,1)$, the true
solution $\boldsymbol{u}\in \big(H^1(\Omega)\big)^3$ is given by
\[
\boldsymbol{u}(x, y, z) =
\left(\begin{array}{c}
\sin(\pi x)\cos(\pi y)
\\
-\sin(\pi y)\cos(\pi x)
\\
0
\end{array}
\right)
+
\left(\begin{array}{c}
x
\\
y
\\
z
\end{array}
\right).
\]
The performance of the PDWG finite element solution for this test problem is illustrated in
Table \ref{table:ex1}. It can be seen that $\mathbf{e}_{\boldsymbol{u}}$, $(e_{\lambda},
\mathbf{e}_{\boldsymbol{q}})$, and $e_{s}$ all have optimal rate of convergence for $k=0$ as showen in Theorem \ref{THM:ErrorEstimate4uh}. The plot of the true solution and the PDWG approximation can be found in Figure \ref{fig:ex1}.

\begin{table}[htbp]\caption{Errors and corresponding rates of convergence for
Example 1.}
\label{table:ex1}
\centering
\begin{tabular}{|c |c | c | c | c | c | c |}
\hline
 $1/h$ &  $\|\varepsilon^{1/2} \mathbf{e}_{\boldsymbol{u}}\|$ & rate &
          $\tnorm{(e_{\lambda}, \mathbf{e}_{\boldsymbol{q}})}$  & rate
          & $\tnorm{e_{s}}$ & rate
\\
\hline
2& 1.64e-1 &      --   &2.95e-1&     --   &3.96e-2&       --
\\
4&  8.16e-2  & 1.01  & 1.68e-1 &  0.82  & 1.86e-2 &  1.09
\\
8& 3.93e-2   & 1.03  & 8.72e-2  & 0.88   &8.79e-3  & 1.09
\\
16& 1.93e-2  & 1.03 &  4.41e-2   &0.98  & 4.48e-3   &0.98
\\
\hline
\end{tabular}
\end{table}

\begin{figure}[h]
  \centering
\begin{subfigure}[b]{0.3\linewidth}
    \centering
    \includegraphics[width=0.9\textwidth]{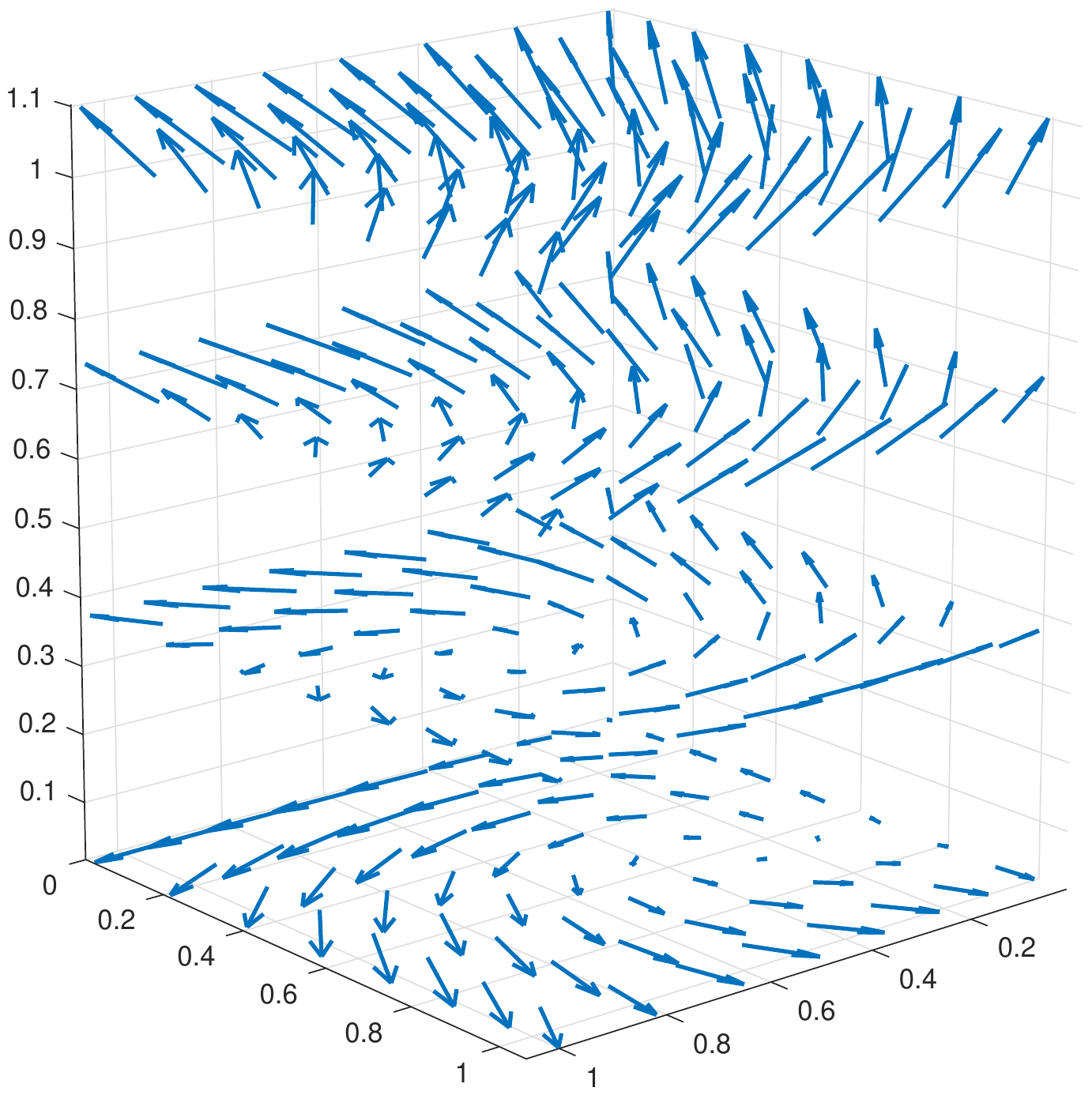}
    \caption{\label{fig:ex1-uexact}}
\end{subfigure}%
\;
\begin{subfigure}[b]{0.3\linewidth}
      \centering
      \includegraphics[width=0.9\textwidth]{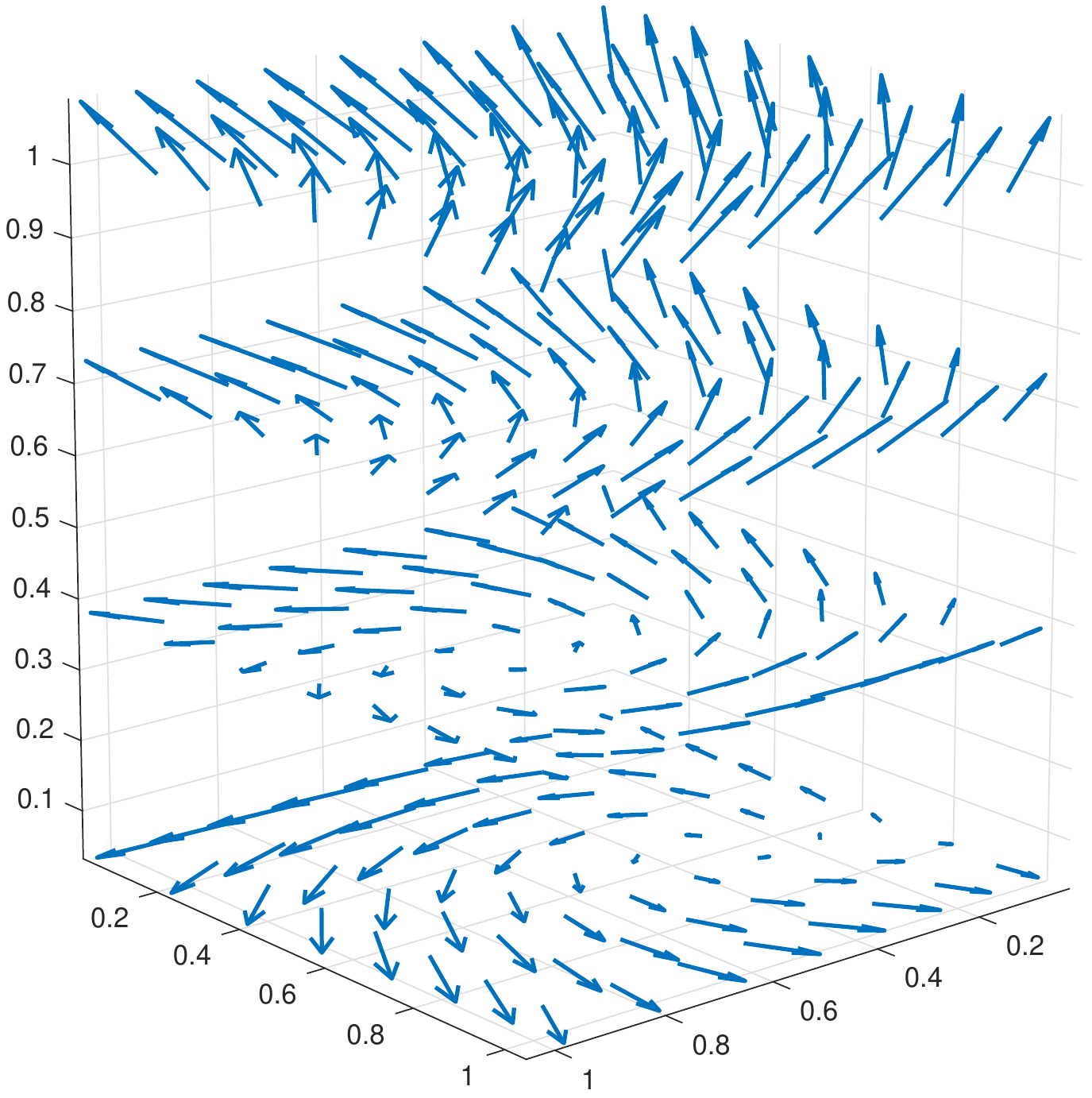}
      \caption{\label{fig:ex1-uwg}}
\end{subfigure}
\;
\begin{subfigure}[b]{0.35\linewidth}
      \centering
      \includegraphics[width=0.9\textwidth]{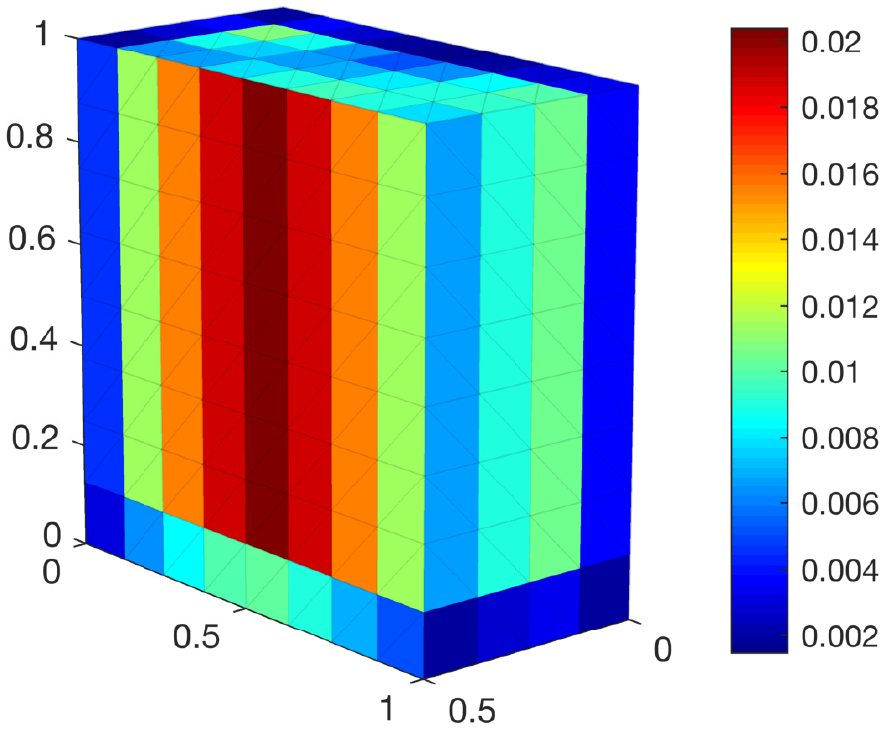}
      \caption{\label{fig:ex1-error}}
\end{subfigure}
\caption{The true solution vector field shown in (\subref{fig:ex1-uexact}) of
example \ref{subsec:ex1} versus the WG approximation (\subref{fig:ex1-uwg}). The
vector fields are plotted on four $z=c$ planes. The distribution of
$\|\varepsilon^{1/2} \mathbf{e}_{\boldsymbol{u}}\|_{T}$ locally is plotted in
(\subref{fig:ex1-error}) on the cut plane $x=1/2$ with $h=1/8$. }
\label{fig:ex1}
\end{figure}

\subsection{Example 2}
\label{subsec:ex2}
The second example is adopted from \cite{li2018weak} with $\varepsilon = I$ and a
singular solution in $\big(H^{1+\frac{2}{3}-\epsilon}(\Omega)\big)^3$:
\[
\boldsymbol{u}(x, y, z)=
\left(\begin{array}{c}
{x(1-x) }
\\
{y(1-y) }
\\
{r^{2 / 3} \sin (2 \theta)z(1-z)}
\end{array}
\right),
\]
in which $r = \sqrt{x^2+y^2}$, and $\theta = \arctan(y/x)+ c$ in the cylindrical
coordinates. Similar to example 1, the result in Table \ref{table:ex2} shows
optimal rates of convergence for $\mathbf{e}_{\boldsymbol{u}}$, 
while slightly sub-optimal in the coarsest two levels for $(e_{\lambda}, \mathbf{e}_{\boldsymbol{q}})$, and $e_{s}$, respectively.  
The plot of the true solution and the PDWG approximation can be found in Figure \ref{fig:ex2}.

%

\begin{table}[htbp]\caption{Errors and corresponding rates of convergence for
Example 2.}
\label{table:ex2}
\centering
\begin{tabular}{| c |c | c | c | c | c | c |}
\hline
 $1/h$ &  $\|\varepsilon^{1/2} \mathbf{e}_{\boldsymbol{u}}\|$ & rate &
          $\tnorm{(e_{\lambda}, \mathbf{e}_{\boldsymbol{q}})}$  & rate
          & $\tnorm{e_{s}}$ & rate
\\
\hline
2& 1.13e-1   &  --   & 2.07e-1  &  --    & 1.74e-2 &   --
\\
4& 5.20e-2   & 1.12  & 1.20e-1  &  0.78  & 1.05e-2 &  0.73
\\
8& 2.50e-2   & 1.05  & 6.27e-2  &  0.94  & 5.53e-3 &  0.93
\\
16& 1.23e-2  & 1.02  & 3.18e-2  &  0.98  & 2.82e-3 &  0.97
\\
\hline
\end{tabular}
\end{table}

\begin{figure}[h]
  \centering
\begin{subfigure}[b]{0.3\linewidth}
    \centering
    \includegraphics[width=0.9\textwidth]{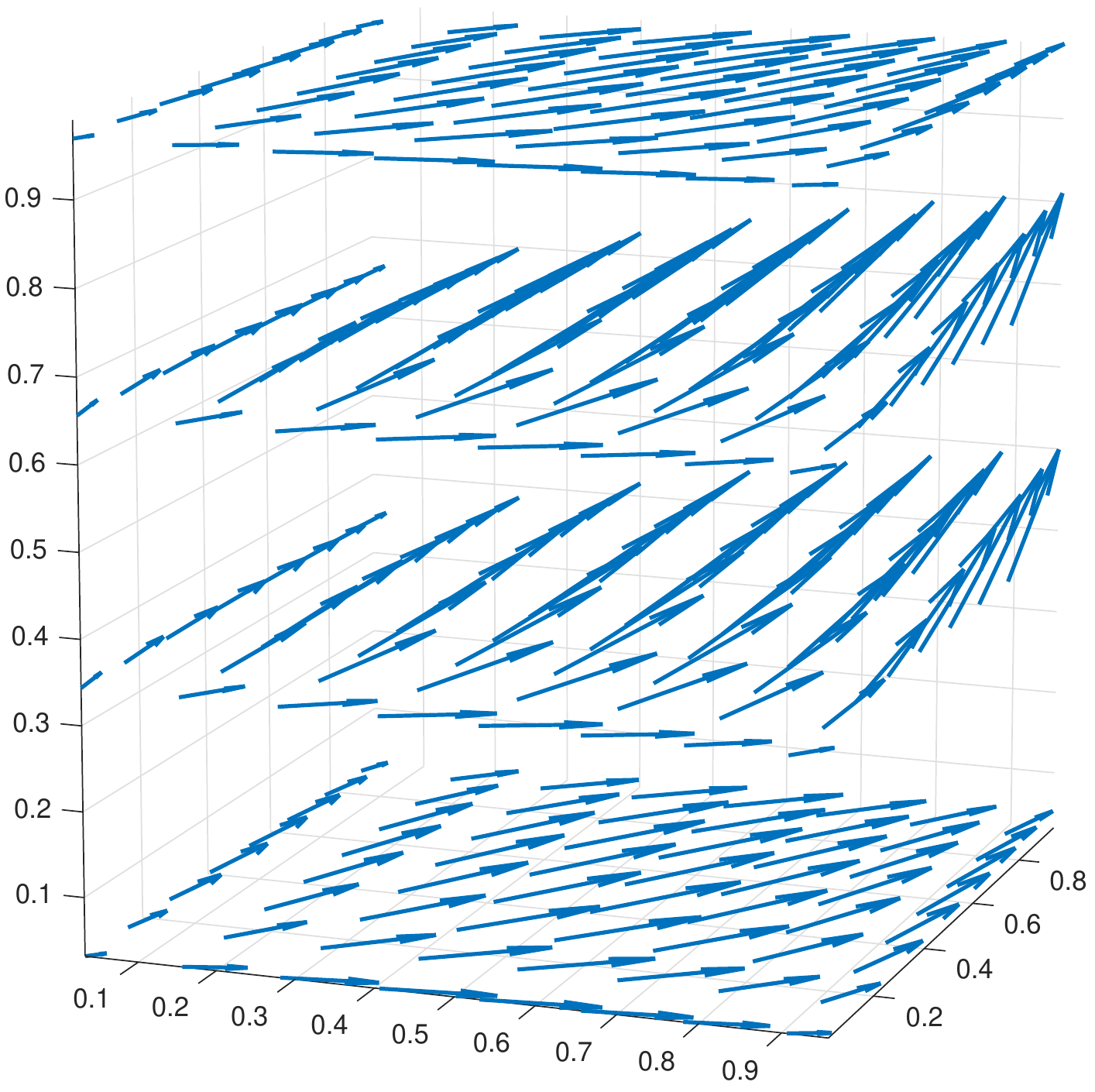}
    \caption{\label{fig:ex2-uexact}}
\end{subfigure}%
\;
\begin{subfigure}[b]{0.3\linewidth}
      \centering
      \includegraphics[width=0.9\textwidth]{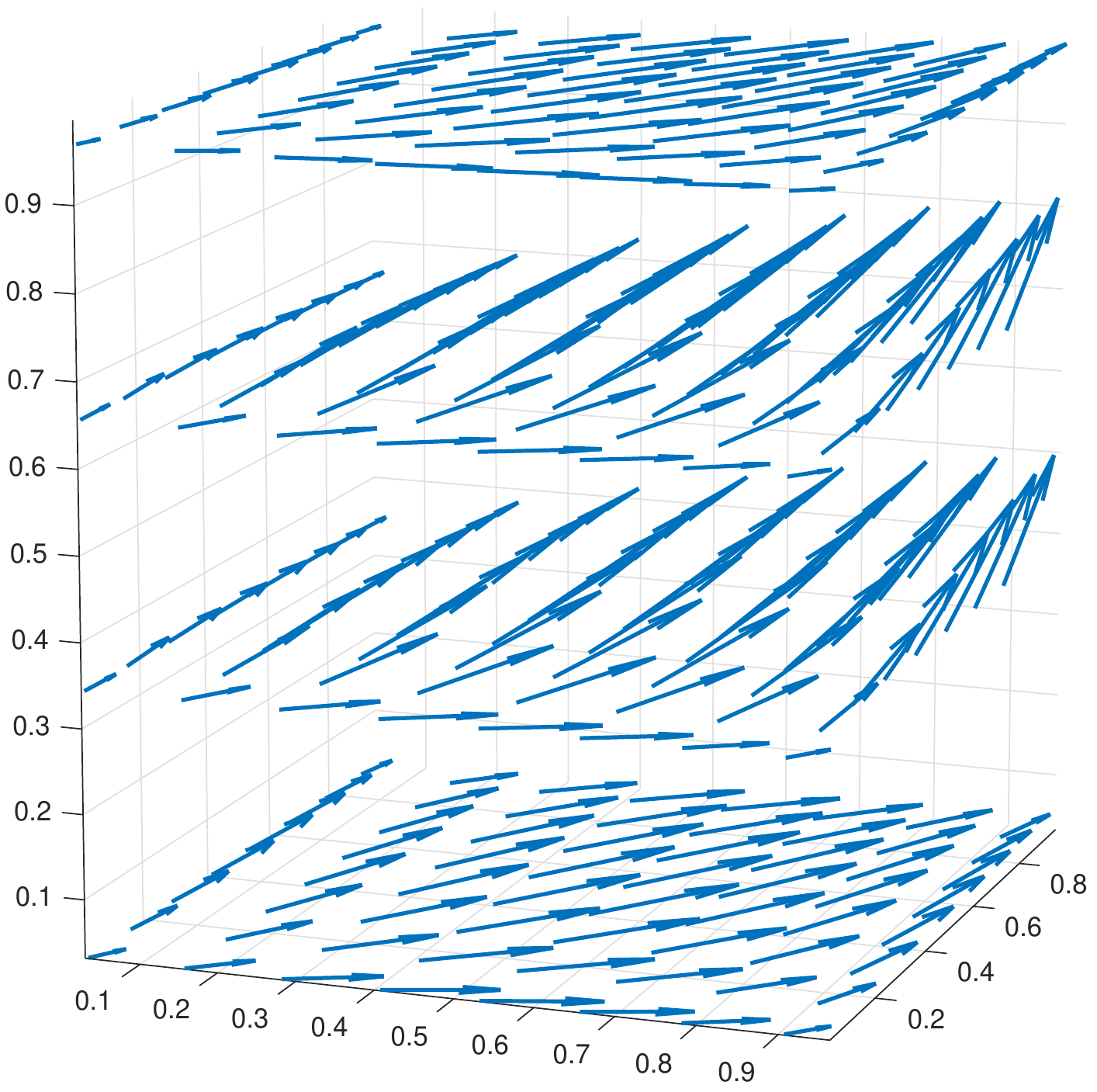}
      \caption{\label{fig:ex2-uwg}}
\end{subfigure}
\;
\begin{subfigure}[b]{0.35\linewidth}
      \centering
      \includegraphics[width=0.9\textwidth]{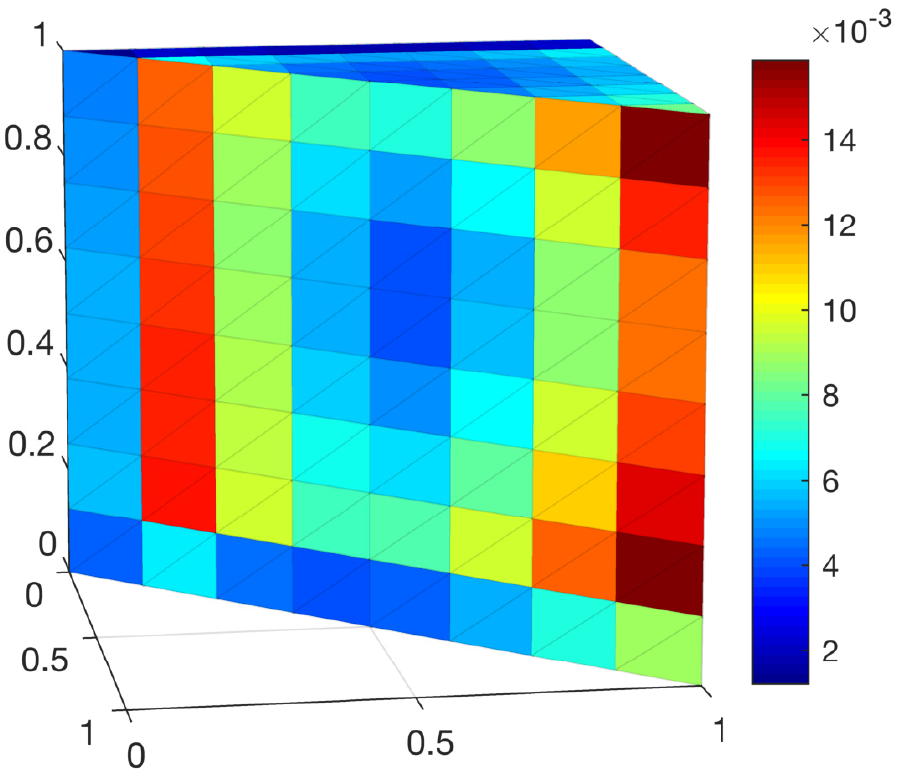}
      \caption{\label{fig:ex2-error}}
\end{subfigure}
\caption{The true solution vector field shown in (\subref{fig:ex2-uexact}) of
example \ref{subsec:ex2} versus the PDWG approximation (\subref{fig:ex2-uwg}). The
vector fields are plotted on four $z=c$ planes. The distribution of
$\|\varepsilon^{1/2} \mathbf{e}_{\boldsymbol{u}}\|_{T}$ locally is plotted in
(\subref{fig:ex2-error}) on the cut plane $x=y$ with $h=1/8$. }
\label{fig:ex2}
\end{figure}

\subsection{Example 3}
\label{subsec:ex3}
This test problem is defined on $\Omega=(-1,1)^2 \times (0,1) \backslash
[0,1] \times[-1,0] \times[0,1]$ with $\varepsilon = I$ and the singular solution in
$\big(H^{2/3-\epsilon}(\Omega) \big)^3$:
\[
\boldsymbol{u} = \nabla\times \left\langle 0, 0,  r^{2/3}
\sin \Big(\frac{2}{3} \theta\Big)\right\rangle.
\]
%
The true solution $\boldsymbol{u}$ is a solenoidal vector
field (see Figure \ref{fig:ex3-u}). Similar to example 2, $r = \sqrt{x^2+y^2}$ and $\theta = \arctan(y/x)+ c$ are the cylindrical coordinates, where $c$ is chosen such that $\boldsymbol{u}\in
{H}(div)\cap H(curl)$. Since $\boldsymbol{u}(x,y,z)$ has unbounded
derivatives as $(x,y,z)$ approaches $\{x=0, y=0\}\cap \partial \Omega$, the Gaussian
quadrature would yield large error on elements with boundary intersecting
$z$-axis. To overcome this difficulty, we replace the true solution by its $L^2$-projection in the error computation:
\[
\|\varepsilon^{1/2} \mathbf{e}_{\mathcalQ_h\boldsymbol{u}}\|:=
\|\varepsilon^{1/2}(\mathcalQ_h\boldsymbol{u} - \boldsymbol{u}_h)\|.
\]
It is observed that $\mathbf{e}_{\mathcalQ_h\boldsymbol{u}}$, $(e_{\lambda},
\mathbf{e}_{\boldsymbol{q}})$, and $e_{s}$ all have maximum possible rate of convergence ($\approx h^{2/3}$) on the finer meshes (see Table \ref{table:ex3}). The plot of the
true solution and the PDWG approximation can be found in Figure \ref{fig:ex3}.

\begin{figure}[h]
  \centering
\begin{subfigure}[b]{0.4\linewidth}
    \centering
    \includegraphics[width=0.9\textwidth]{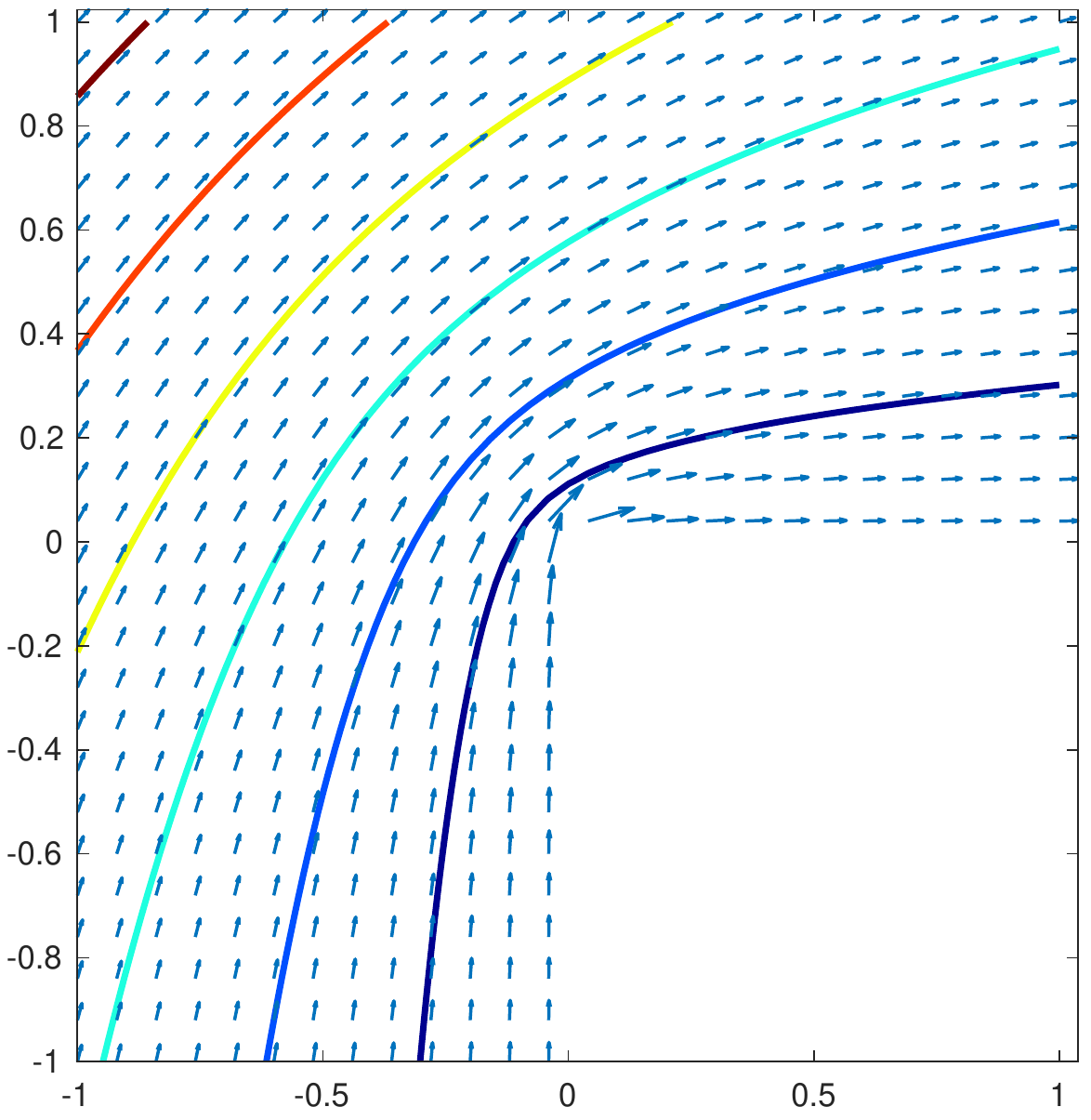}
    \caption{\label{fig:ex3-uexact}}
\end{subfigure}%
\hspace{0.3in}
\begin{subfigure}[b]{0.45\linewidth}
      \centering
      \includegraphics[width=0.9\textwidth]{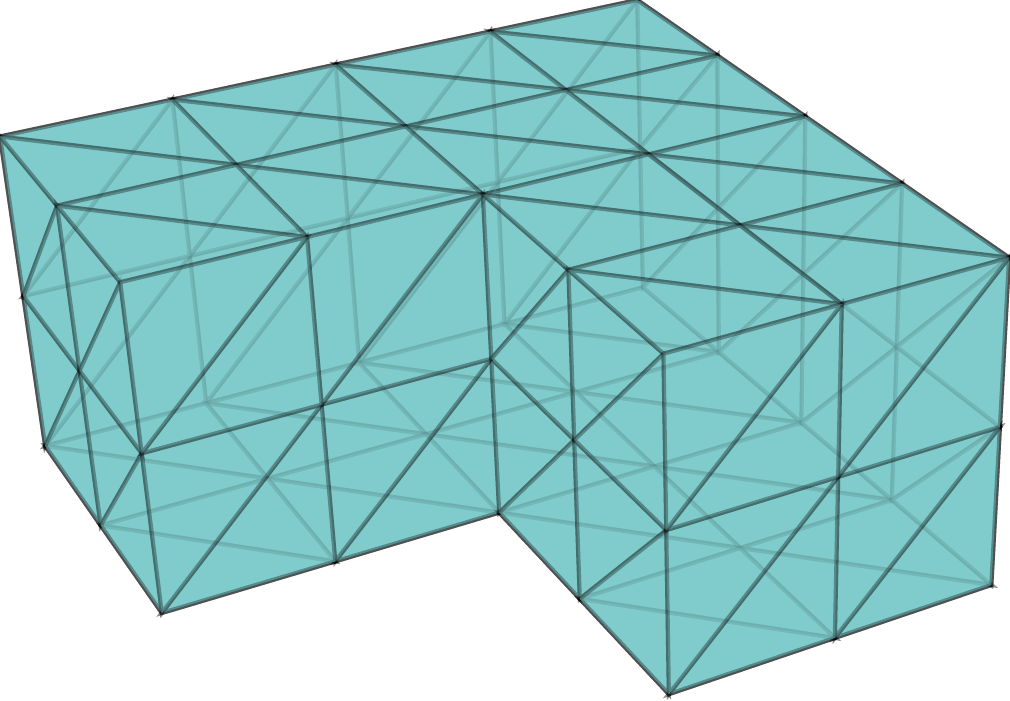}
      \caption{\label{fig:ex3-mesh}}
\end{subfigure}
\caption{The true solution vector field shown in (\subref{fig:ex3-uexact}) of
example \ref{subsec:ex3} view from above on $z=1/4$ plane together with the level
set of its $z$-component. A coarse mesh ($h=1/2$)
used in example \ref{subsec:ex3} is illustrated in (\subref{fig:ex3-mesh}). }
\label{fig:ex3-u}
\end{figure}

\begin{table}[htbp]\caption{Errors and corresponding rates of convergence for
Example 3.}
\label{table:ex3}
\centering
\begin{tabular}{| c |c | c | c | c | c | c |}
\hline
 $1/h$ &  $\|\varepsilon^{1/2} \mathbf{e}_{\mathcalQ_h\boldsymbol{u}}\|$ & rate &
          $\tnorm{(e_{\lambda}, \mathbf{e}_{\boldsymbol{q}})}$  & rate
          & $\tnorm{e_{s}}$ & rate
\\
\hline
2& 5.29e-2   &  --   & 3.35e-1  &  --    & 4.99e-2 &   --
\\
4& 3.13e-2   & 0.75  & 2.19e-1  &  0.61  & 3.30e-2 &  0.60
\\
8& 1.91e-2   & 0.72  & 1.41e-1  &  0.63  & 2.14e-2 &  0.62
\\
16& 1.16e-2  & 0.72  & 9.03e-2  &  0.65  & 1.37e-2 &  0.64
\\
\hline
\end{tabular}
\end{table}

\begin{figure}[h]
  \centering
\begin{subfigure}[b]{0.3\linewidth}
    \centering
    \includegraphics[width=0.9\textwidth]{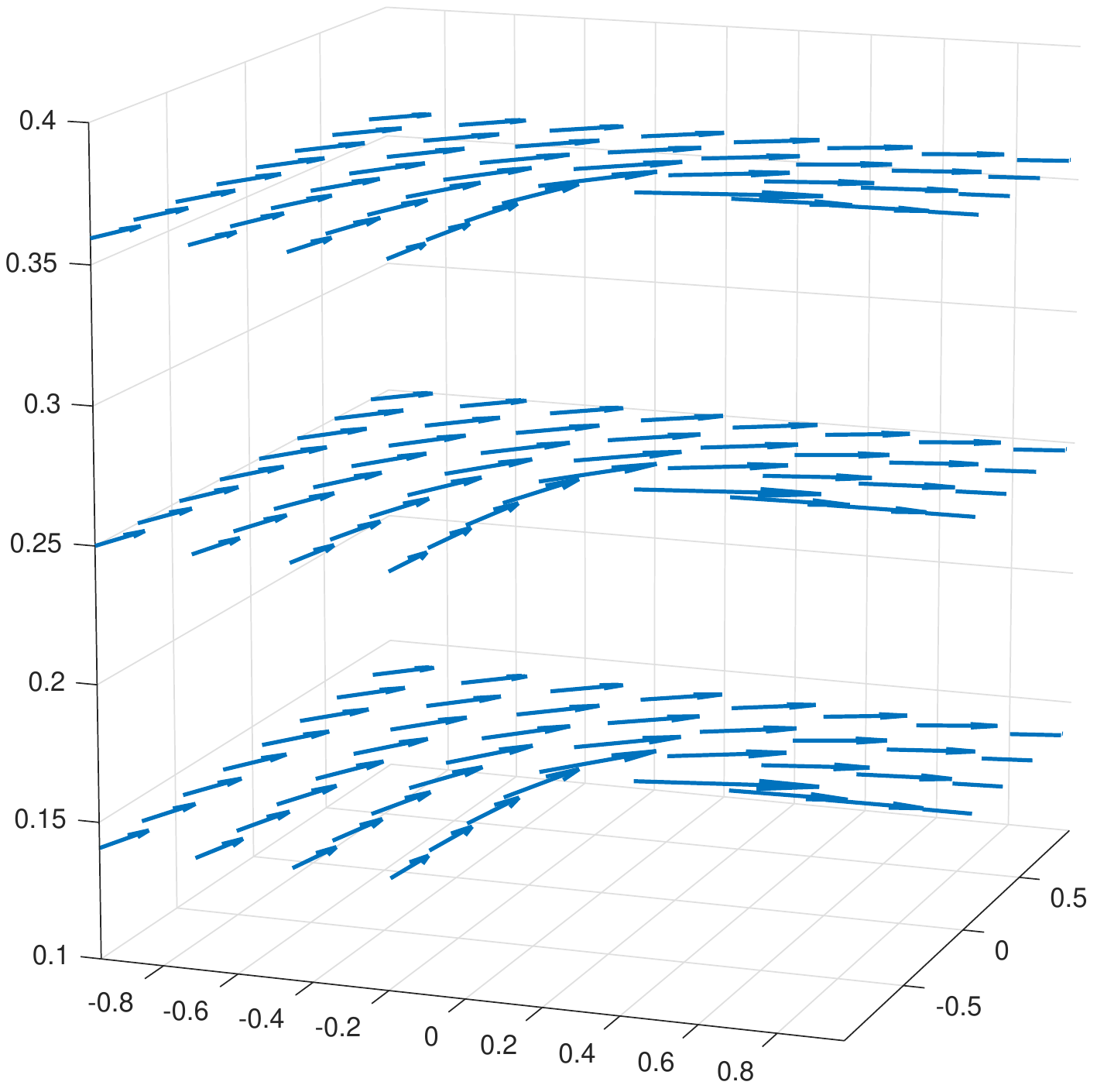}
    \caption{\label{fig:ex3-Qu}}
\end{subfigure}%
\;
\begin{subfigure}[b]{0.3\linewidth}
      \centering
      \includegraphics[width=0.9\textwidth]{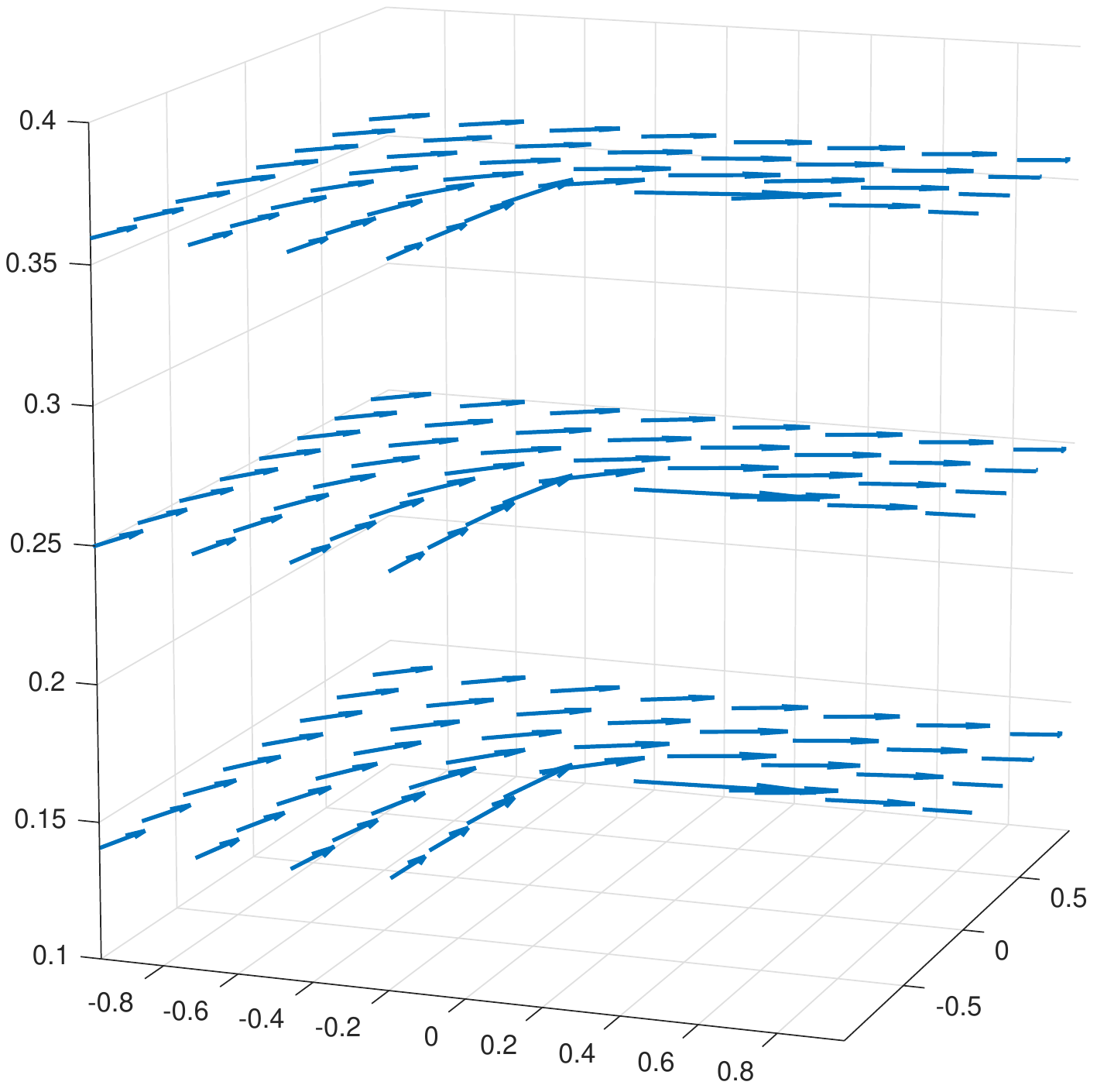}
      \caption{\label{fig:ex3-uwg}}
\end{subfigure}
\;
\begin{subfigure}[b]{0.35\linewidth}
      \centering
      \includegraphics[width=0.9\textwidth]{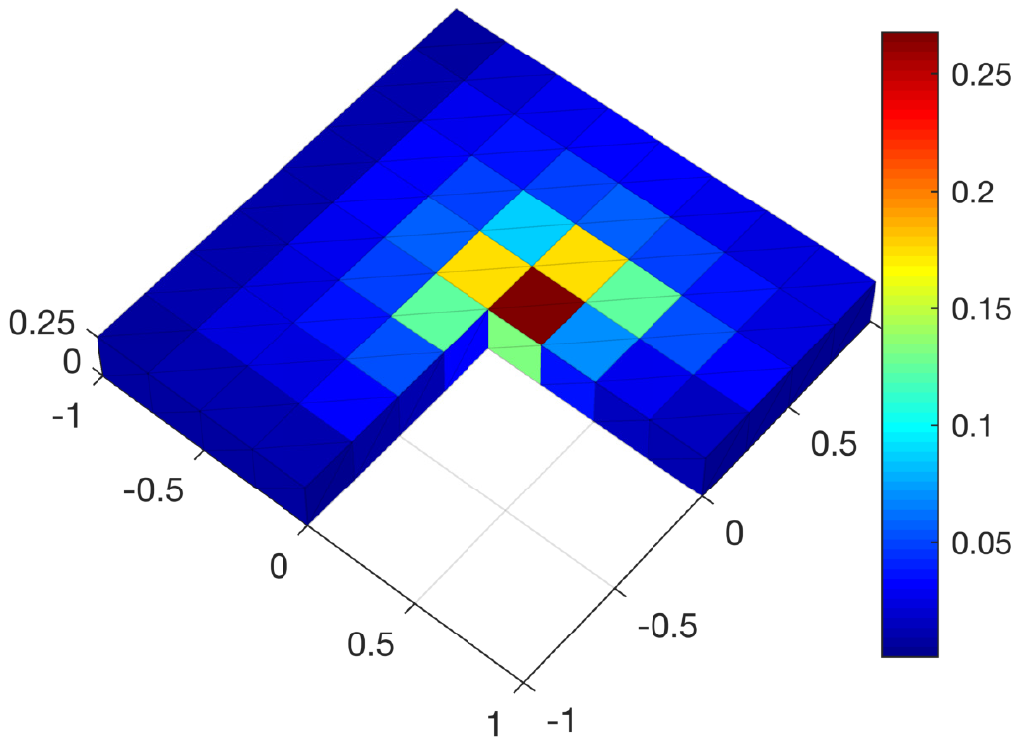}
      \caption{\label{fig:ex3-error}}
\end{subfigure}
\caption{The vector field of $\mathcalQ_h\boldsymbol{u}$ is shown in (\subref{fig:ex3-Qu}) of
example \ref{subsec:ex3} versus the PDWG approximation (\subref{fig:ex3-uwg}). The
vector fields are plotted on several $z=c$ planes. The distribution of
$\|\varepsilon^{1/2} \mathbf{e}_{\mathcalQ_h\boldsymbol{u}}\|_{T}$ locally is plotted in
(\subref{fig:ex3-error}) on the cut plane $z=1/4$ when $h=1/8$. }
\label{fig:ex3}
\end{figure}

\subsection{Example 4}
\label{subsec:ex4}
This test problem is defined on the domain $\Omega = (-3/2, 1/2)^3\backslash [-1,0]^3$
such that the domain boundary $\partial \Omega$ consits of two disjoint
connected components $\Gamma_0 = \partial (-3/2, 1/2)^3$ and
$\Gamma_1 = \partial (-1, 0)^3$. The true solution is given by
\[
\boldsymbol{u} = \nabla(r^{1/6}), \quad \text{ with } \quad r =\sqrt{x^2+y^2+z^2}.
\]
It is straightforward to see that $\boldsymbol{u}$ is singular near $(0,0,0)$ which
is one of the nonconvex corners of the cavity (see Figure \ref{fig:ex4-u}), 
and $\boldsymbol{u}\in [H^{2/3-\epsilon}(\Omega)]^3$.

\begin{figure}[h]
  \centering
\begin{subfigure}[b]{0.4\linewidth}
    \centering
    \includegraphics[width=0.8\textwidth]{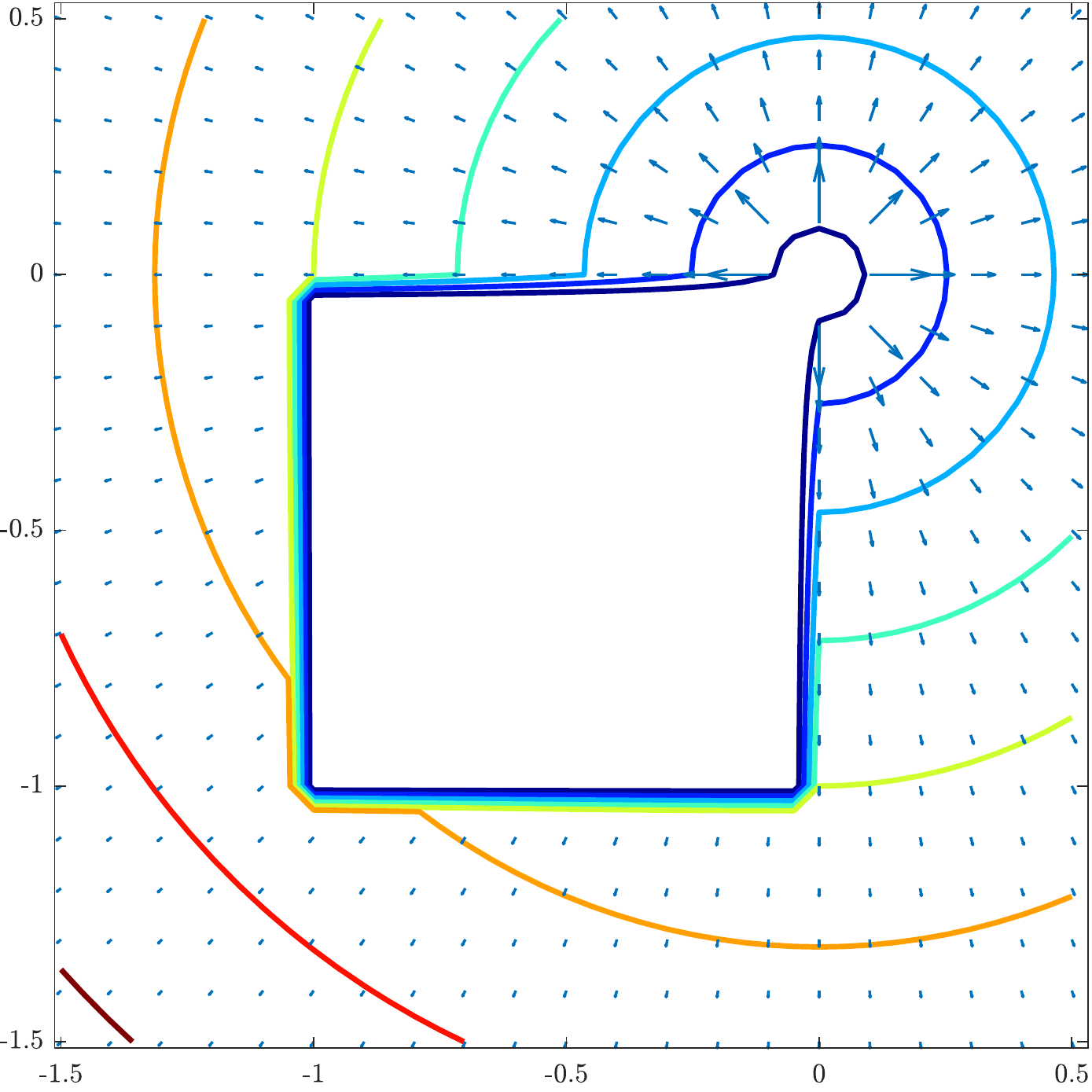}
    \caption{\label{fig:ex4-uexact}}
\end{subfigure}%
\hspace{0.3in}
\begin{subfigure}[b]{0.4\linewidth}
      \centering
      \includegraphics[width=0.8\textwidth]{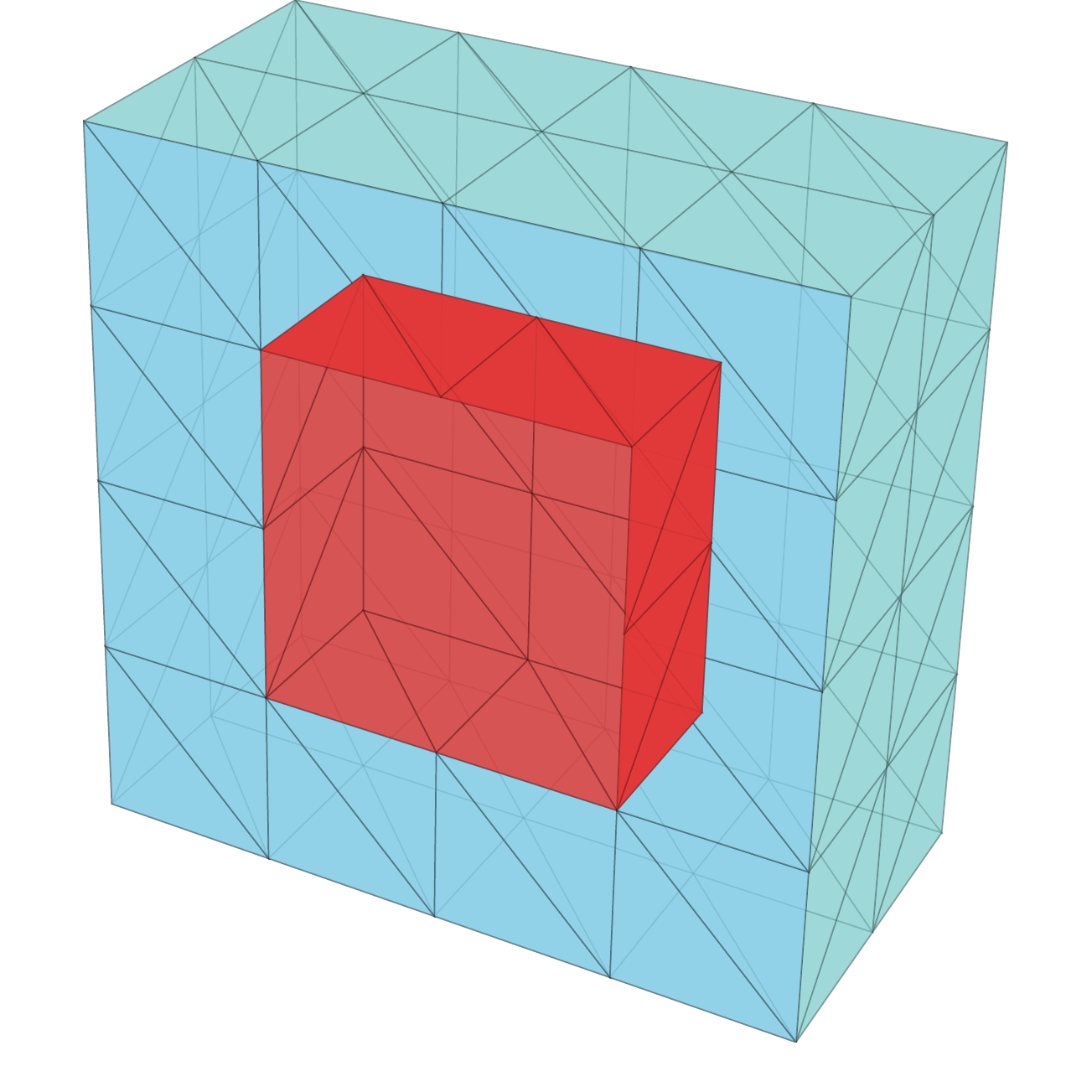}
      \caption{\label{fig:ex4-mesh}}
\end{subfigure}
\caption{(\subref{fig:ex4-uexact}): plot of the true solution field of
example \ref{subsec:ex4}, view from above on $z=0$ plane. (\subref{fig:ex4-mesh}): a coarse mesh ($h=1/2$) used in example \ref{subsec:ex4} cut by the plane $x=0$, with boundary faces of $\Gamma_1$ being highlighted in red. }
\label{fig:ex4-u}
\end{figure}

Table \ref{table:ex4} illustrates the convergence of the PDWG method for example \ref{subsec:ex4}.
It can be seen that $\mathbf{e}_{\mathcalQ_h\boldsymbol{u}}$ has the optimal rate of convergence of $O(h^{2/3})$. On the other hand, the convergence for $(e_{\lambda}, \mathbf{e}_{\boldsymbol{q}})$ and $e_{s}$ seems to be approximating the optimal rate in this numerical test.

For $s = \{s_0, s_b\}$, when solving the algebraic system, we seek
a constant $c_1$ such that $s_b|_{\Gamma_1}=c_1$ through a simple post-processing by treating $s_b|_{\Gamma_1}$ as the fixed DoFs first. Denote by $\mathbf{U}$
the vector representation of the solution $(\boldsymbol{u}_h, \lambda_h, s_h,
\boldsymbol{q}_h)$,
and $\mathbf{U}_s = c_1 \mathbf{S}$, $\mathbf{S}= (0, \cdots, 1, \cdots, 1, \cdots, 0)$, the indicator vector representing a solution with $s_b=1$ on
$\Gamma_1$ while all other DoFs are zero. Let $A$ be the full stiffness
matrix including from all nodal
bases (including boundary faces), while $A^{\rm int}$ be the stiffness matrix for all the free DoFs: including the interior DoFs for $\boldsymbol{u}_h$, $s_h$, and $\boldsymbol{q}_h$, all except 1 fixed DoF for $\lambda_h$.
Let $R$ be the restriction operator such that $R:
\mathbf{U}\mapsto \mathbf{U}^{\rm int}$, which is the vector representing all the aforementioned free DoFs. First we solve the following algebraic system:
\[
A^{\rm int} \mathbf{U}^{\rm int} = R \mathbf{F},
\]
where $\mathbf{F}$ is the full vector of the right hand side. Then the constant $c_1$ is sought by solving the following minimization problem:
\[
c_1 = \operatornamewithlimits{argmin}_{a\in \mathbb{R}}
\| A(\mathbf{U}^{\rm int} + a\mathbf{S}) - F \|_{\ell^2}.
\]
The plot of the projection of the true solution and the PDWG approximation is shown in
Figure \ref{fig:ex4}.

\begin{figure}[h]
  \centering
\begin{subfigure}[b]{0.3\linewidth}
    \centering
    \includegraphics[width=0.9\textwidth]{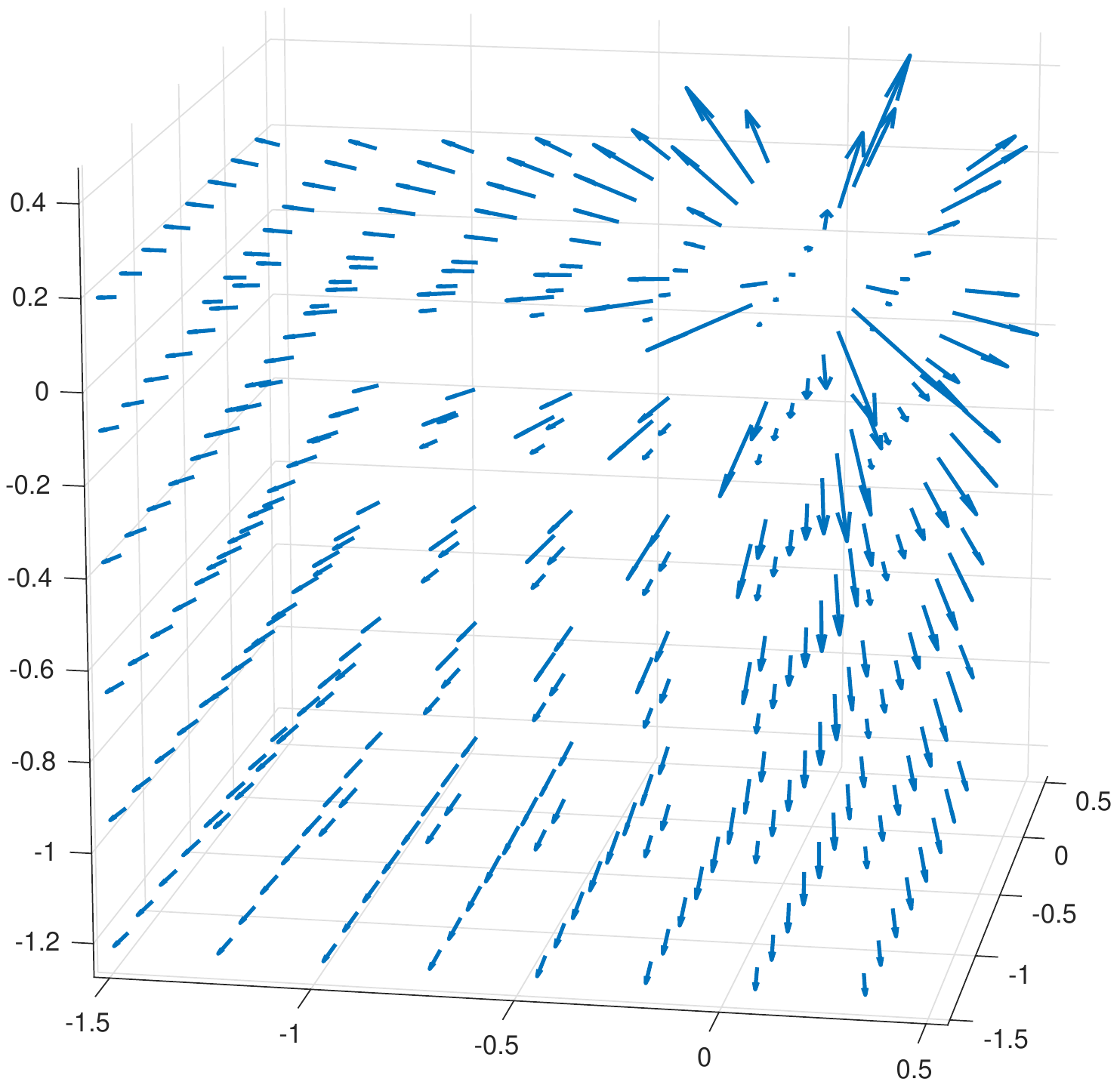}
    \caption{\label{fig:ex4-Qu}}
\end{subfigure}%
\;
\begin{subfigure}[b]{0.3\linewidth}
      \centering
      \includegraphics[width=0.9\textwidth]{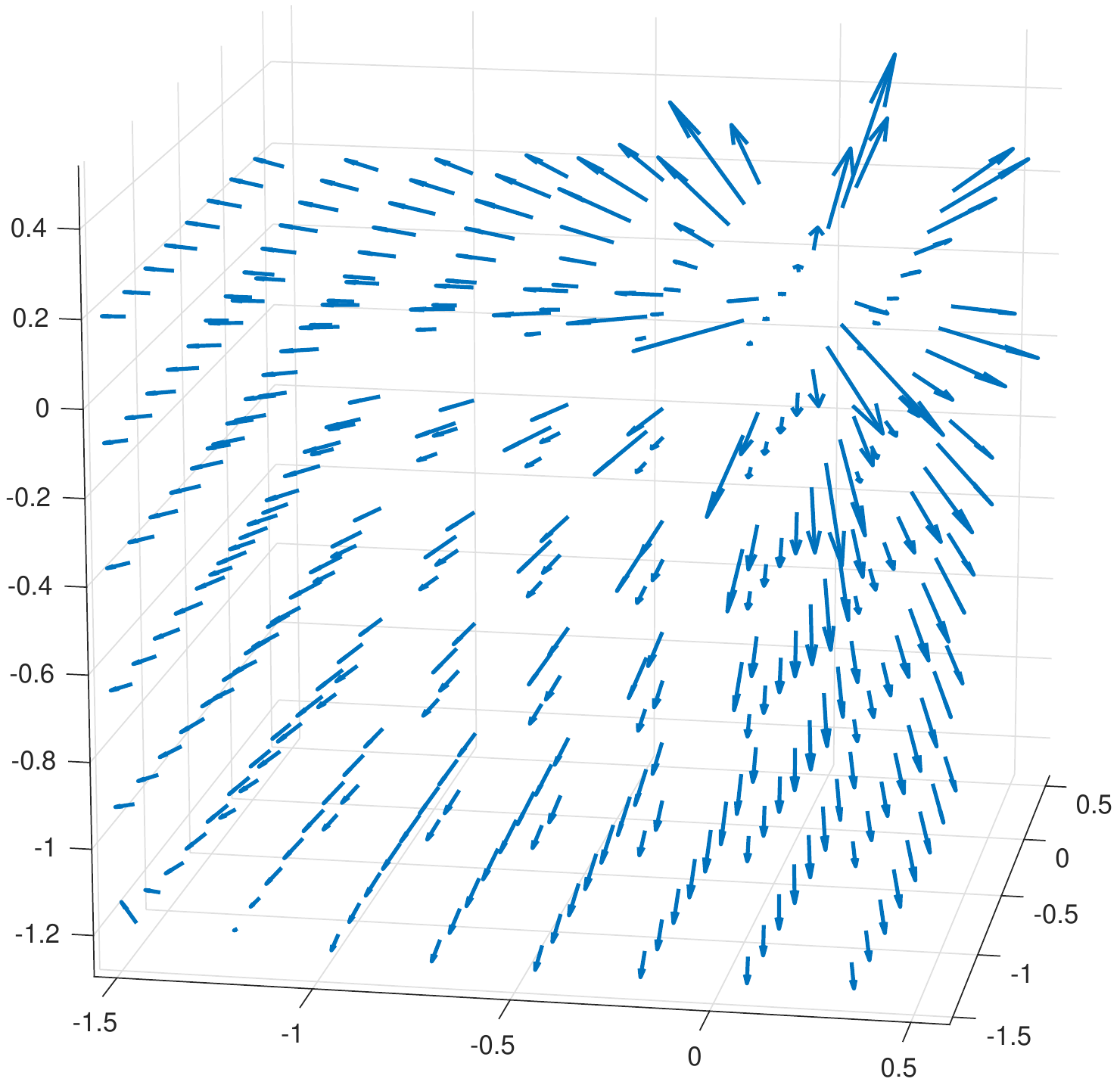}
      \caption{\label{fig:ex4-uwg}}
\end{subfigure}
\;
\begin{subfigure}[b]{0.35\linewidth}
      \centering
      \includegraphics[width=0.9\textwidth]{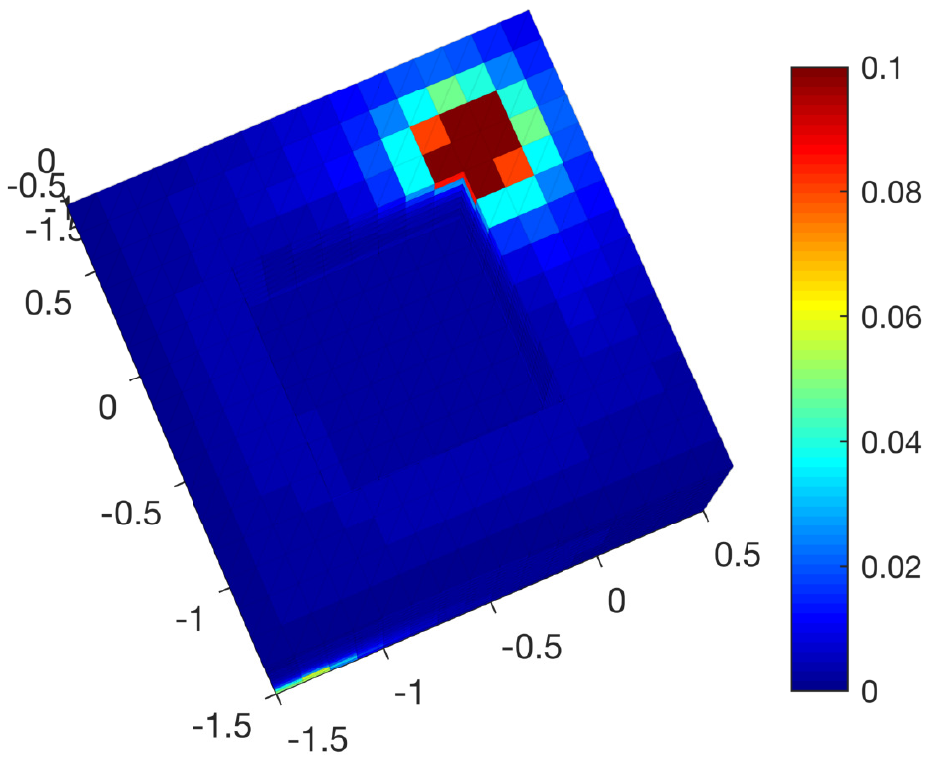}
      \caption{\label{fig:ex4-error}}
\end{subfigure}
\caption{Example \ref{subsec:ex4}: The vector field $\mathcalQ_h\boldsymbol{u}$ is shown in
(\subref{fig:ex4-Qu}) for
Example \ref{subsec:ex4} versus the PDWG approximation (\subref{fig:ex4-uwg}). The
vector fields are plotted on several $z=c$ planes. The error distribution of
$\|\varepsilon^{1/2} \mathbf{e}_{\mathcalQ_h\boldsymbol{u}}\|_{T}$ is plotted in
(\subref{fig:ex4-error}) on the cut plane $z=0$ with meshsize $h=1/8$. }
\label{fig:ex4}
\end{figure}

\begin{table}[htbp]\caption{Errors and corresponding rates of convergence for
Example 4.}
\label{table:ex4}
\centering
\begin{tabular}{| c |c | c | c | c | c | c |}
\hline
 $1/h$ &  $\|\varepsilon^{1/2} \mathbf{e}_{\mathcalQ_h\boldsymbol{u}}\|$ & rate &
          $\tnorm{(e_{\lambda}, \mathbf{e}_{\boldsymbol{q}})}$  & rate
          & $\tnorm{e_{s}}$ & rate
\\
\hline
2& 1.90-1   &  --   & 2.51e-1  &  --    & 2.04e-2 &   --
\\
4& 1.23e-1   & 0.63  & 1.93e-1  &  0.38  & 1.69e-2 &  0.27
\\
8& 7.78e-2   & 0.66  & 1.35e-1  &  0.51  & 1.24e-2 &  0.44
\\
16& 4.91e-2  & 0.66  & 9.03e-2  &  0.58  & 8.47e-3 &  0.55 
\\
\hline
\end{tabular}
\end{table}

\subsection{Example 5}
\label{subsec:ex5}
In this example we consider singular solutions in the
following vector potential form on a toroidal 
domain with 1 hole: $\Omega = \Big( (-1,\frac12)^2\backslash
[-\frac12,0]^2\times [0,\frac12] \Big)\times (0,\frac12)$
\[
\boldsymbol{u} = \nabla\times \left\langle 0, 0,  r^{\gamma}
\sin \big(\alpha \theta\big)\right\rangle,
\]
where $r$ and $\theta$ are the cylindrical coordinates defined as in Example \ref{subsec:ex3}. It can be verified that for $\gamma\neq 1$,
$\boldsymbol{u}\in \big(H^{\gamma-\epsilon}(\Omega)\big)^3$, and for $\gamma<1$,
this vector field is singular near a non-convex corner centered at $z$-axis (see Figure
\ref{fig:ex5-u}).

\begin{figure}[h]
  \centering
\begin{subfigure}[b]{0.4\linewidth}
    \centering
    \includegraphics[width=0.8\textwidth]{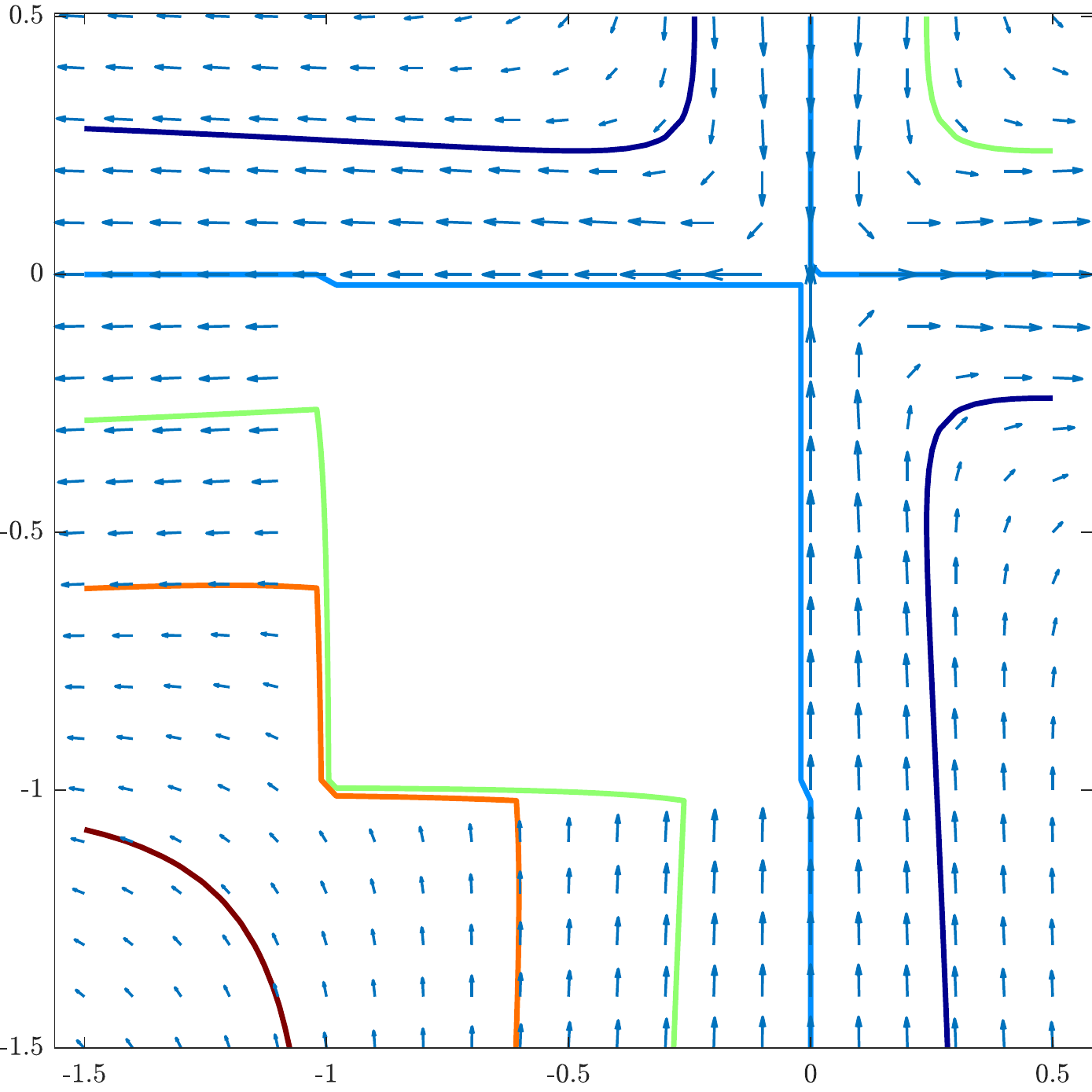}
    \caption{\label{fig:ex5-uexact}}
\end{subfigure}%
\hspace{0.3in}
\begin{subfigure}[b]{0.4\linewidth}
      \centering
      \includegraphics[width=0.8\textwidth]{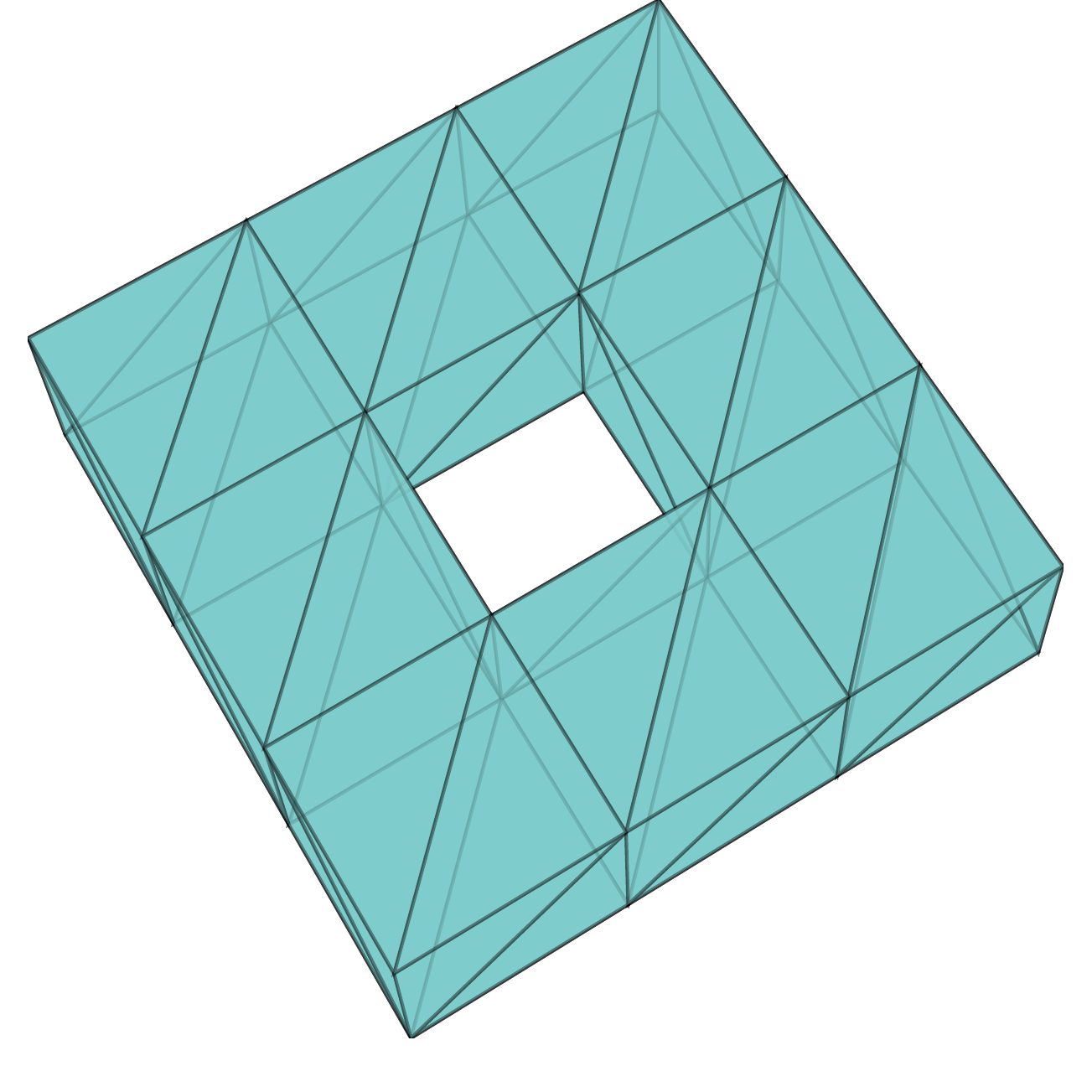}
      \caption{\label{fig:ex5-mesh}}
\end{subfigure}
\caption{Example \ref{subsec:ex5}: Plot of the singular true solution vector field with
$\gamma=3/4$ and $\alpha=2$ is shown in (\subref{fig:ex5-uexact}),  view from above on
$z=0$ plane. The level set contours are for the $z$-component of the vector
potential. The coarse mesh ($h=1/2$) is shown in (\subref{fig:ex5-mesh}).}
\label{fig:ex5-u}
\end{figure}

Unlike Example \ref{subsec:ex3} in which $\gamma=\alpha$ was set on an L-shaped domain, we choose $\alpha=2$ in this example, so that the resulting vector field is non-harmonic for $\gamma\neq \alpha$, and
\[
\nabla\times \boldsymbol{u} = \left\langle 0, 0,  -\Delta\left(r^{\gamma}
\sin \big(\alpha \theta\big)\right)\right\rangle =
 \left\langle 0, 0,
 \left(\alpha^{2}-\gamma^{2}\right)\left(x^{2}+y^{2}\right)^{\gamma / 2-1} \sin
\left(\alpha \theta\right)\right\rangle.
\]
Consequently, $\nabla\times \boldsymbol{u}\notin L^2(\Omega)$ if $\gamma\leq 1$.
Nevertheless, due to the unique nature of the PDWG method, we still obtain noteworthy
convergence result for the case of $\gamma\leq 1$. In Table \ref{table:ex5},
we have compiled several cases ranging from smooth to singular and plotted the PDWG
approximation vs $\mathcalQ_h\boldsymbol{u}$ in Figure \ref{fig:ex5}.

\begin{itemize}
\item {\em Regular case $\gamma =
1.25$:} \ $\mathbf{e}_{\boldsymbol{u}}$ and $e_{s}$ show the optimal rates of convergence at $O(h)$, while $\mathbf{e}_{\mathcalQ_h\boldsymbol{u}}$ shows a superconvergence.
$(e_{\lambda}, \mathbf{e}_{\boldsymbol{q}})$ only shows a slightly suboptimal
rate of convergence during first two refinements, and optimal thereafter.

\item {\em Singular case $\gamma = 1$:}\  In this case, we have $\boldsymbol{u}\in
H^{1-\epsilon}_{loc}(\Omega)$ and $\boldsymbol{u}\notin
\boldsymbol{H}(\mathbf{curl})$.
$\mathbf{e}_{\boldsymbol{u}}$ shows a rate of
convergence at $O(h^{0.9})$ asymptotically. Like the smooth case,
$\mathbf{e}_{\mathcalQ_h\boldsymbol{u}}$ shows a superconvergence with rates higher than $1$. $(e_{\lambda}, \mathbf{e}_{\boldsymbol{q}})$ and $e_{s}$ show an optimal rate of convergence asymptotically.

\item {\em Singular case $\gamma = 2/3$:} \ In this case, one has $\boldsymbol{u}\in
H^{2/3-\epsilon}(\Omega)$, and $\boldsymbol{u}\notin \boldsymbol{H}(\mathbf{curl})$.
Both $\mathbf{e}_{\boldsymbol{u}}$ and $e_{s}$ 
show optimal rate of convergence at $O(h^{2/3})$. Similarly to previous two
cases, $\mathbf{e}_{\mathcalQ_h\boldsymbol{u}}$ exhibits superconvergence with a rate of 
$O(h^{0.9})$. $(e_{\lambda}, \mathbf{e}_{\boldsymbol{q}})$ shows a convergence with a slightly suboptimal rate of $O(h^{0.6})$.

\end{itemize}

\begin{table}[htb]\caption{Errors and the rates of convergence with different
$\gamma$ for Example 5}
\label{table:ex5}
\centering
\begin{tabular}{| c |c |c | c | c | c | c | c |c | c|}
        \hline
 & $1/h$ & $\|\mathbf{e}_{\boldsymbol{u}}\|$  & rate
         &  $\|\varepsilon^{1/2} \mathbf{e}_{\mathcalQ_h\boldsymbol{u}}\|$ & rate
         &$\tnorm{(e_{\lambda}, \mathbf{e}_{\boldsymbol{q}})}$  & rate
          & $\tnorm{e_{s}}$ & rate
\\
       \hline
             &2& 3.96e-1 &  -- & 1.62e-1 &  -- & 9.01e-1 &  -- & 8.23e-2 &  --
 \\
$\gamma=5/4$ &4& 2.09e-1 & 0.92 & 7.69e-2 & 1.07 & 4.94e-1 & 0.87 & 5.12e-2 & 0.68
\\
 (smooth)    &8& 1.07e-1 & 0.96 & 3.23e-2 & 1.25 & 2.63e-1 & 0.91 & 2.70e-2 & 0.93
\\
            &16& 5.44e-2 & 0.98 & 1.26e-3 & 1.35 & 1.37e-1 & 0.94 & 1.32e-2 & 1.03
\\
\hline
 & $1/h$ & $\|\mathbf{e}_{\boldsymbol{u}}\|$  & rate
         &  $\|\varepsilon^{1/2} \mathbf{e}_{\mathcalQ_h\boldsymbol{u}}\|$ & rate
         &$\tnorm{(e_{\lambda}, \mathbf{e}_{\boldsymbol{q}})}$  & rate
          & $\tnorm{e_{s}}$ & rate
\\
 \hline
            &2& 5.34e-1 & --  &2.46e-1 &  --  & 1.20e0 &  --  & 1.42e-1 &  --
 \\
$\gamma=1$  &4& 3.06e-1 & 0.80 & 1.28e-1 & 0.94 & 7.11e-1 & 0.76 & 8.95e-2 & 0.66
\\
(singular)  &8& 1.67e-1 & 0.87 & 5.58e-2 & 1.19 & 4.04e-1 & 0.81 & 4.94e-2 & 0.86
\\
           &16& 8.82e-2 & 0.88 & 2.55e-2 & 1.13 & 2.25e-1 & 0.85 & 2.56e-2 & 0.95
  \\
 \hline

  & $1/h$ & $\|\mathbf{e}_{\boldsymbol{u}}\|$  & rate
          &  $\|\varepsilon^{1/2} \mathbf{e}_{\mathcalQ_h\boldsymbol{u}}\|$ & rate
          &$\tnorm{(e_{\lambda}, \mathbf{e}_{\boldsymbol{q}})}$  & rate
           & $\tnorm{e_{s}}$ & rate
 \\
  \hline
            &2& 8.87e-1 & --   & 4.55e-1 &  --  & 2.05e0 &  --   & 3.41e-1 &  --
 \\
$\gamma=2/3$&4& 5.87e-1 & 0.59 & 2.75e-1 & 0.73 & 1.40e0  & 0.55 & 2.39e-1 & 0.51
\\
(singular)  &8& 3.70e-1 & 0.67 & 1.39e-1 & 0.98 & 9.28e-1 & 0.60 & 1.53e-1 & 0.65
\\
           &16& 2.34e-1 & 0.66 & 7.51e-2 & 0.89 & 6.01e-1 & 0.62 & 9.51e-2 & 0.68
  \\
 \hline
 \end{tabular}
\end{table}

\begin{figure}[h]
  \centering
\begin{subfigure}[b]{0.3\linewidth}
    \centering
    \includegraphics[width=0.9\textwidth]{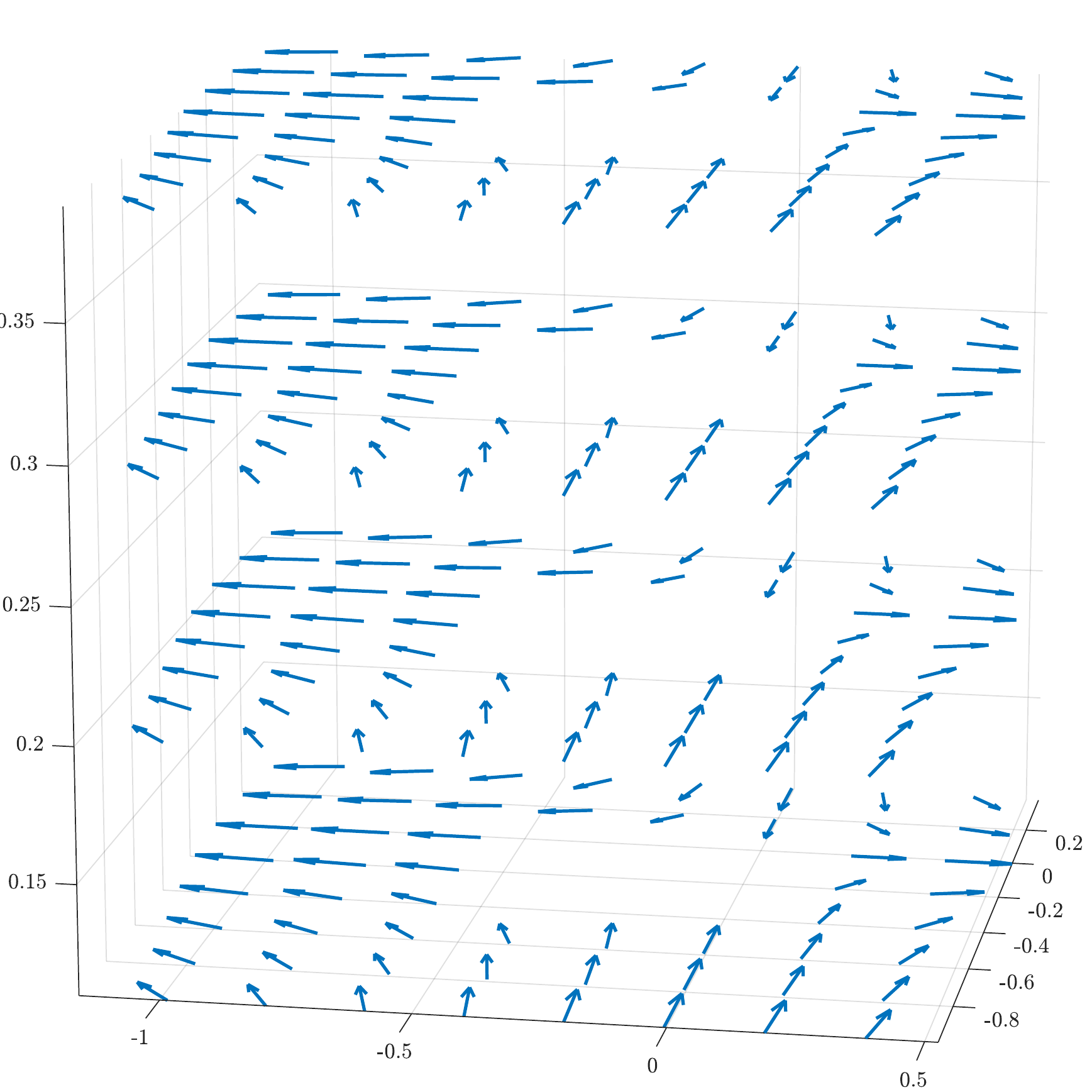}
    \caption{\label{fig:ex5-Qu}}
\end{subfigure}%
\;
\begin{subfigure}[b]{0.3\linewidth}
      \centering
      \includegraphics[width=0.9\textwidth]{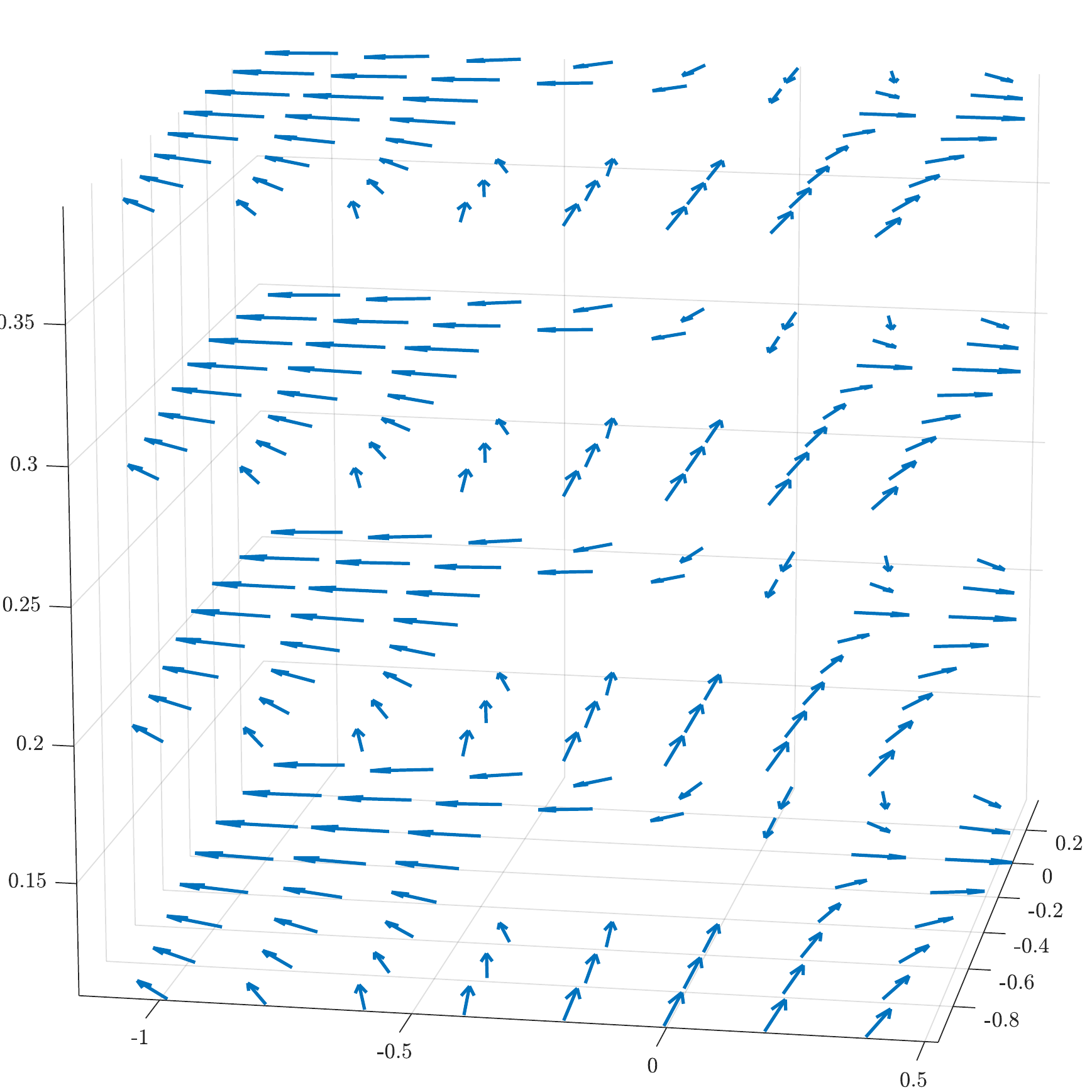}
      \caption{\label{fig:ex5-uwg}}
\end{subfigure}
\;
\begin{subfigure}[b]{0.35\linewidth}
      \centering
      \includegraphics[width=0.9\textwidth]{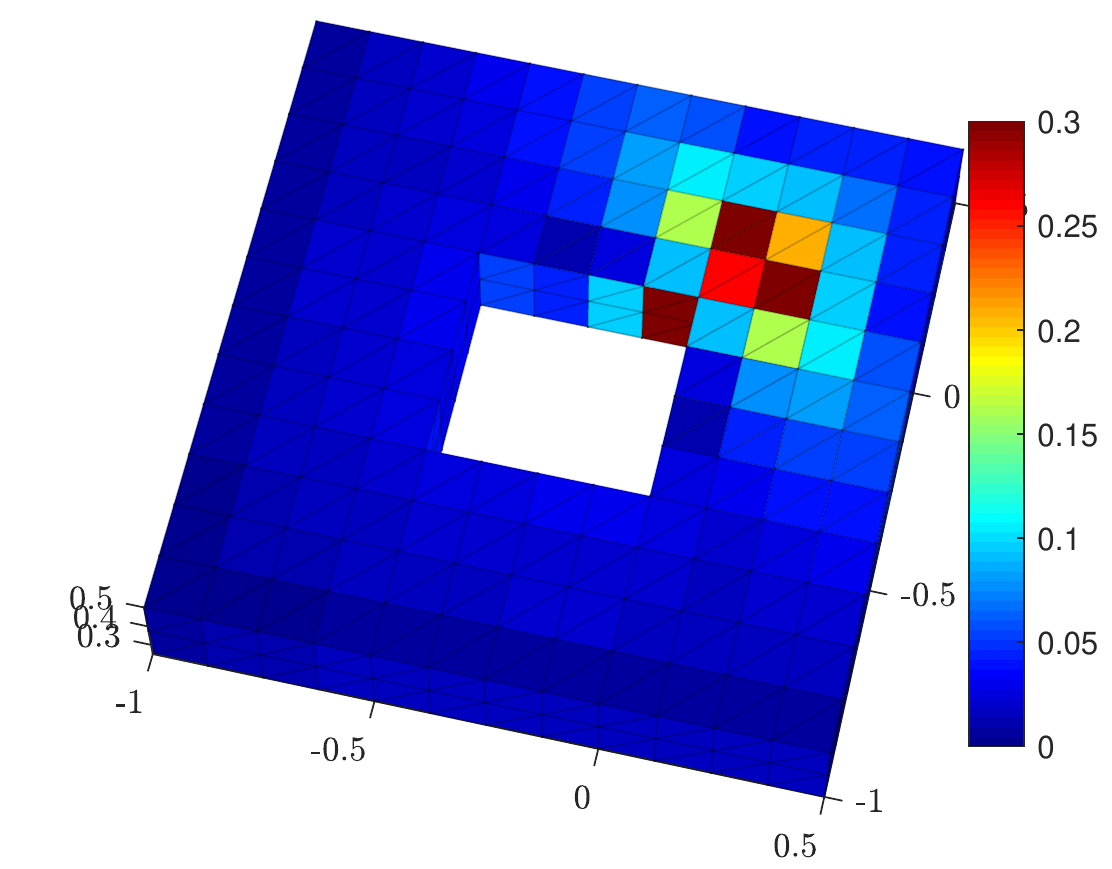}
      \caption{\label{fig:ex5-error}}
\end{subfigure}
\caption{Example \ref{subsec:ex5}: The vector field of $\mathcalQ_h\boldsymbol{u}$ is shown in
(\subref{fig:ex5-Qu}) versus the PDWG approximation (\subref{fig:ex5-uwg}) for
$\gamma=2/3$ (singular case). The vector fields are plotted on several $z=c$ planes.
The distribution of $\|\varepsilon^{1/2} \mathbf{e}_{\boldsymbol{u}}\|_{T}$ locally
is plotted in (\subref{fig:ex5-error}) on the cut plane $z=1/4$ with meshsize $h=1/8$. }
\label{fig:ex5}
\end{figure}

\subsection{Example 6}
\label{subsec:ex6}
In this example we consider a singular solution bearing the same form with that in
Example \ref{subsec:ex5} on a
toroidal domain with 2 holes: $\Omega = \Big[ (-1,\frac32)^2\backslash
\big\{[-\frac12,0]^2\times [0,\frac12]\cup  [\frac12,1]\times[-\frac12,0]\big\}
\Big] \times
[0,\frac12] $
\[
\boldsymbol{u} = \nabla\times \left\langle 0, 0,  r_1^{\gamma_1}
\sin \big(\alpha \theta_1\big) + r_2^{\gamma_2}
\sin \big(\alpha \theta_2\big)\right\rangle,
\]
where $(r_i, \theta_i)$ are the cylindrical coordinates centered at a nonconvex
corner of the $i$-th hole, i.e., $r_1 = \sqrt{x^2+y^2}$, $r_2 =
\sqrt{(x-1)^2+y^2}$, $\theta_1 = \arctan(y/x)+c_1$, and $\theta_2 =
\arctan\big(y/(x-1)\big)+c_2$. In this example, we choose $\gamma_1=1/2$ and
$\gamma_2=2/3$ such
that the vector field is singular near the nonconvex corners of both holes (see
Figure \ref{fig:ex6-u}). In fact, the vector field $\boldsymbol{u}$ behaves as
$H^{1/2-\epsilon}$-regular in a neighborhood of the edge
$\{x=0,y=0\}$, and as $H^{2/3-\epsilon}$-regular in a neighborhood of the edge
$\{x=1,y=0\}$. In this example, one has $\boldsymbol{u}\notin \boldsymbol{H}(\mathbf{curl})$
as in example \ref{subsec:ex5}. The convergence results are shown in Table
\ref{table:ex6}. It can be seen that
$\mathbf{e}_{\boldsymbol{u}}$, $(e_{\lambda}, \mathbf{e}_{\boldsymbol{q}})$,
and $e_{s}$ show optimal rates of convergence at $O(h^{1/2})$. 
The local error is more
prominent near the nonconvex corner of the first hole where $\boldsymbol{u}$ locally
is $H^{1/2-\epsilon}$-regular (see Figure \ref{fig:ex6}). Similarly to previous two
cases, $\mathbf{e}_{\mathcalQ_h\boldsymbol{u}}$ shows superconvergence with a rate approximately at $O(h^{3/4})$.

\begin{figure}[h]
  \centering
\begin{subfigure}[b]{0.6\linewidth}
    \centering
    \includegraphics[width=0.9\textwidth]{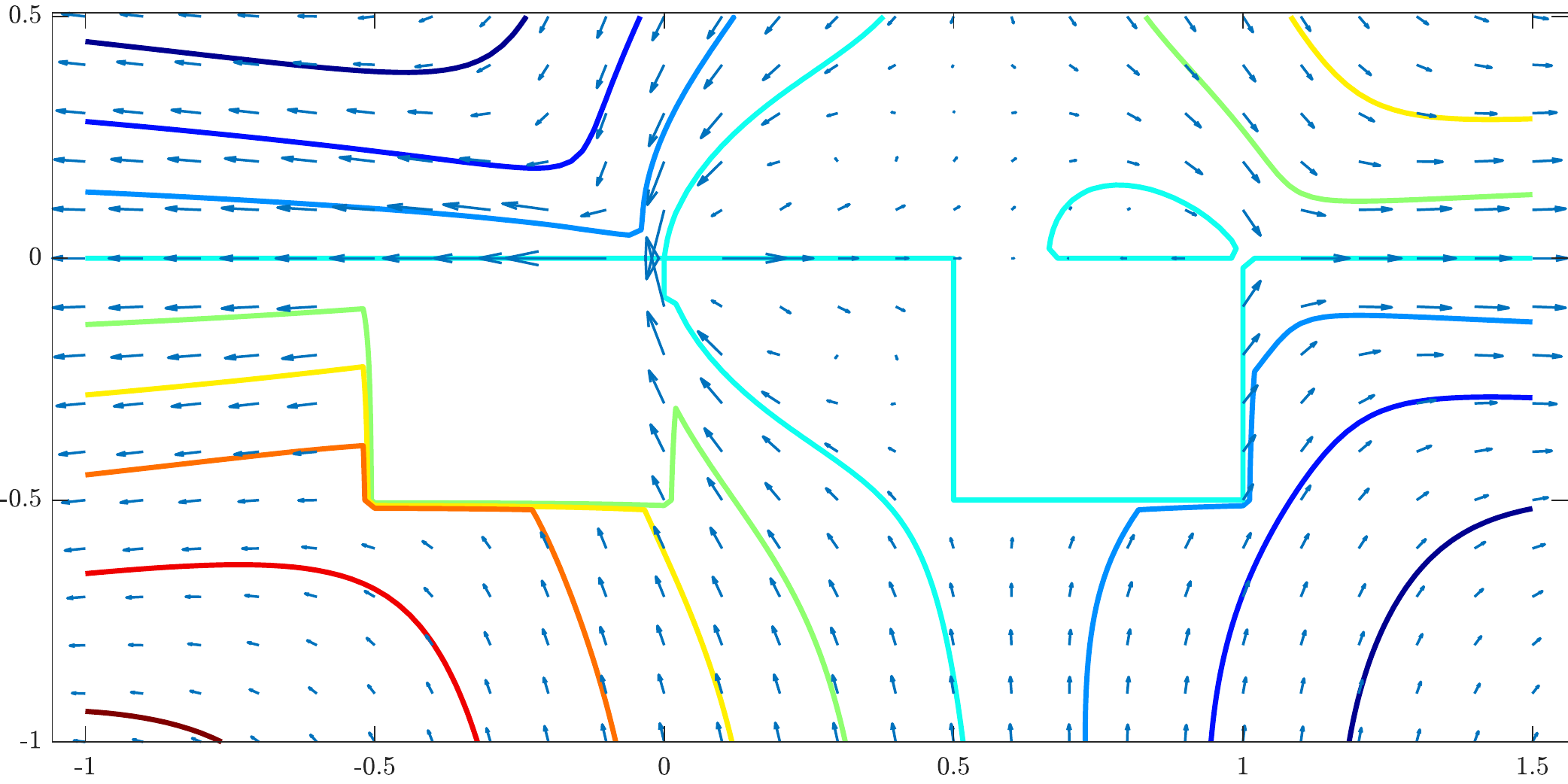}
    \caption{\label{fig:ex6-uexact}}
\end{subfigure}%
\quad
\begin{subfigure}[b]{0.35\linewidth}
      \centering
      \includegraphics[width=0.9\textwidth]{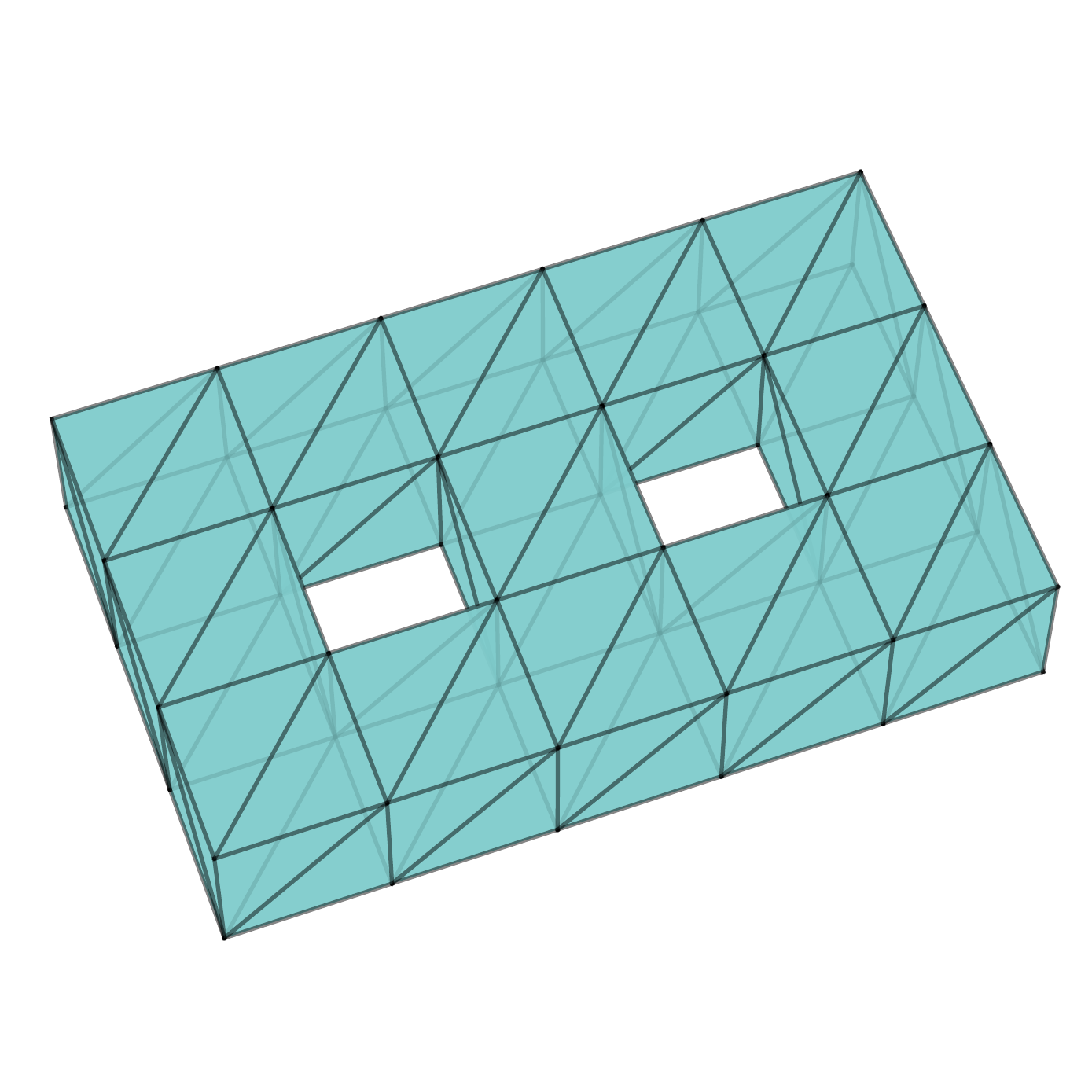}
      \caption{\label{fig:ex6-mesh}}
\end{subfigure}
\caption{Example \ref{subsec:ex6}: the singular true solution vector field when
$\gamma_1=1/2$, $\gamma_2 = 2/3$, and $\alpha=2$ is shown in
(\subref{fig:ex6-uexact}) view from above on
$z=0$ plane. The level set contours are for the $z$-component of the vector
potential used. The coarse mesh ($h=1/2$) is illustrated in (\subref{fig:ex6-mesh}).}
\label{fig:ex6-u}
\end{figure}

\begin{table}[htb]\caption{Errors and the rates of convergence  for Example 6}
\label{table:ex6}
\centering
\begin{tabular}{| c |c |c | c | c | c | c | c |c | c|}
        \hline
 & $1/h$ & $\|\mathbf{e}_{\boldsymbol{u}}\|$  & rate
         &  $\|\varepsilon^{1/2} \mathbf{e}_{\mathcalQ_h\boldsymbol{u}}\|$ & rate
         &$\tnorm{(e_{\lambda}, \mathbf{e}_{\boldsymbol{q}})}$  & rate
          & $\tnorm{e_{s}}$ & rate
\\
       \hline
               &2& 1.49e0  &  --  & 8.37e-1 &  --  & 3.39e0 &  -- & 6.38e-1 &  --
 \\
$\gamma_1=1/2$ &4& 1.04e0  & 0.52 & 5.18e-1 & 0.69 & 2.48e0 & 0.45 & 4.74e-1 &0.43
\\
$\gamma_2=2/3$ &8& 6.99e-1 & 0.57 & 2.84e-1 & 0.87 & 1.77e0 & 0.49 & 3.27e-1 &0.54
\\
              &16& 4.79e-1 & 0.55 & 1.70e-1 & 0.73 & 1.24e0 & 0.51 & 2.22e-1 & 0.55
\\
\hline
 \end{tabular}
\end{table}

\begin{figure}[h]
  \centering
\begin{subfigure}[b]{0.4\linewidth}
    \centering
    \includegraphics[width=0.9\textwidth]{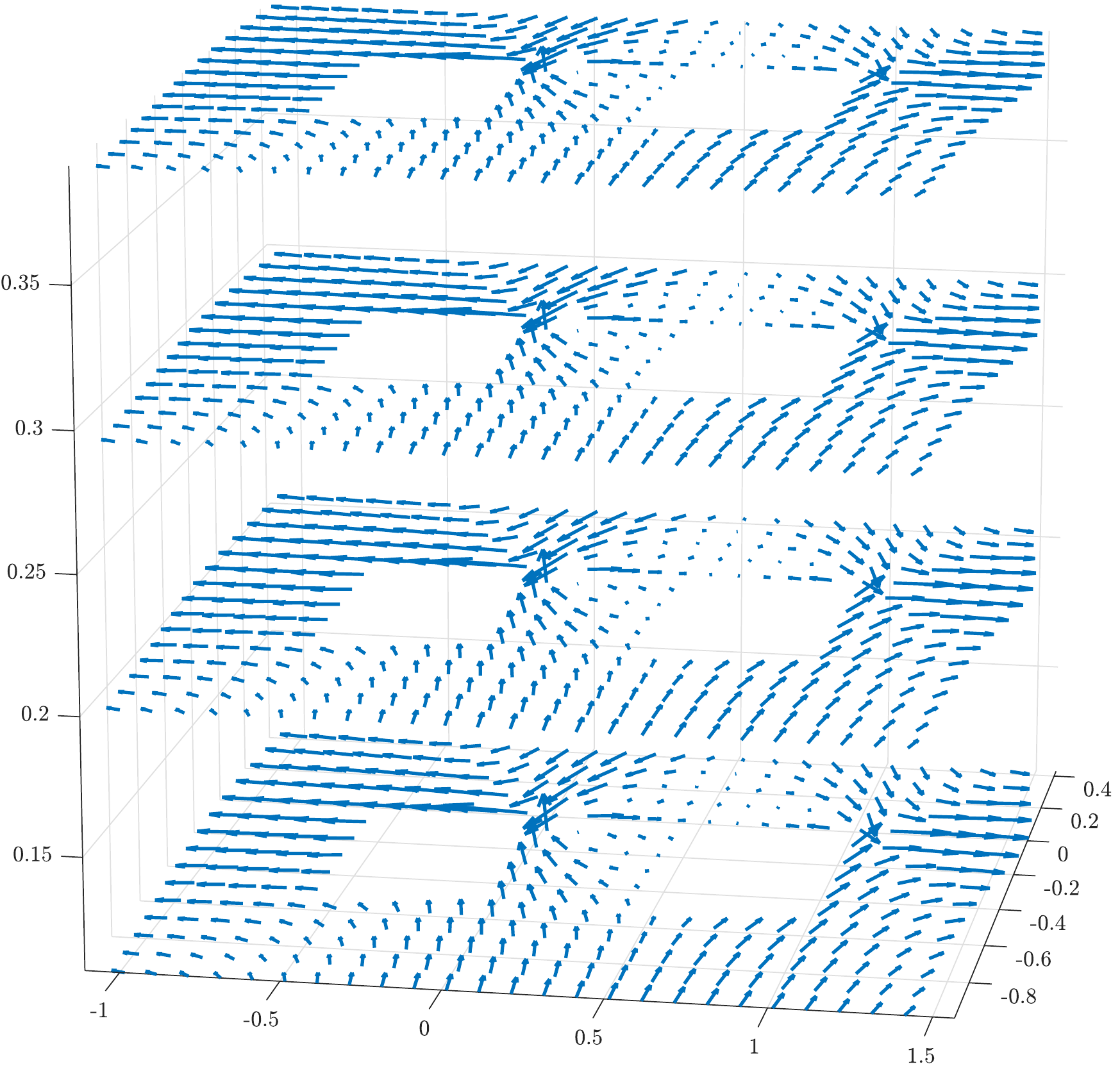}
    \caption{\label{fig:ex6-Qu}}
\end{subfigure}%
\quad
\begin{subfigure}[b]{0.4\linewidth}
      \centering
      \includegraphics[width=0.9\textwidth]{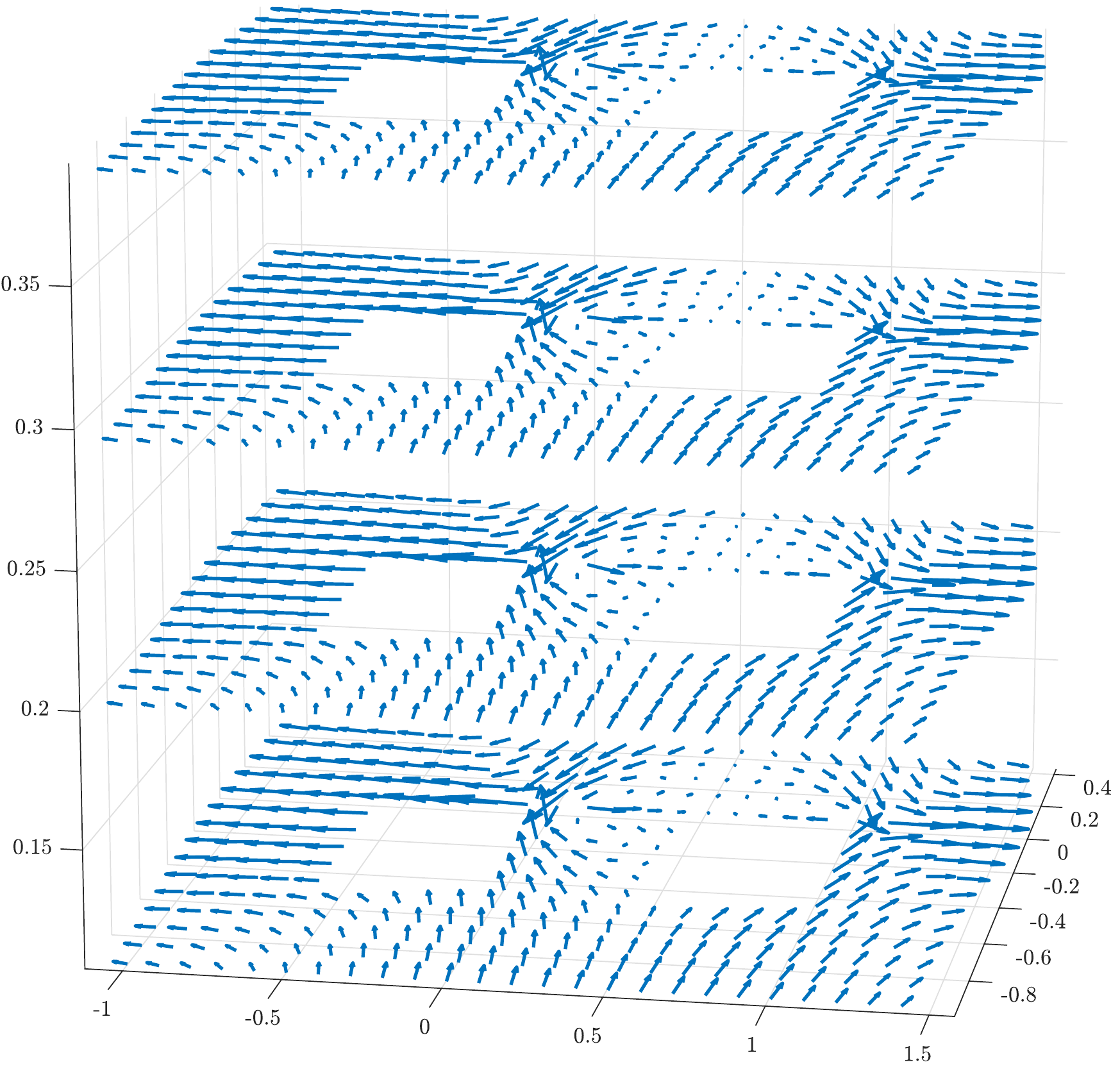}
      \caption{\label{fig:ex6-uwg}}
\end{subfigure}
\quad
\begin{subfigure}[b]{0.7\linewidth}
      \centering
      \includegraphics[width=0.9\textwidth]{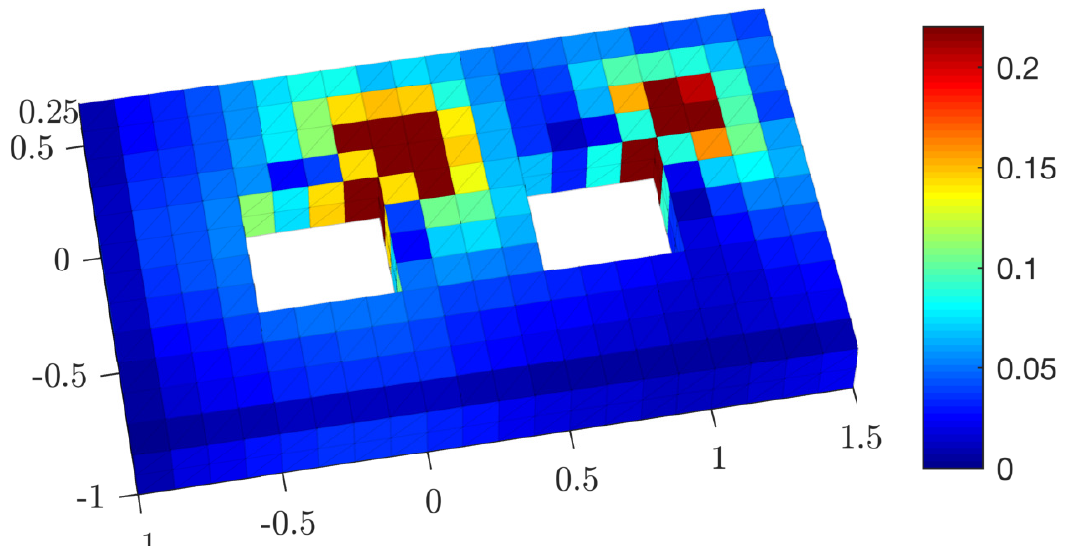}
      \caption{\label{fig:ex6-error}}
\end{subfigure}
\caption{Example \ref{subsec:ex6}: The vector field of $\mathcalQ_h\boldsymbol{u}$ is shown in
(\subref{fig:ex6-Qu}) versus the PDWG approximation (\subref{fig:ex6-uwg}) for
$\gamma_1=1/2$  and $\gamma_2 = 2/3$ (singular case). The vector fields are plotted
on several $z=c$ planes. The distribution of $\|\varepsilon^{1/2}
\mathbf{e}_{\boldsymbol{u}}\|_{T}$ locally
is plotted in (\subref{fig:ex6-error}) on the cut plane $z=1/4$ with meshsize $h=1/8$. }
\label{fig:ex6}
\end{figure}

\subsection{Example 7}
\label{subsec:ex7}
In this example, we report some computational results for a test problem on a toroidal domain with 1 hole. The numerical data does not show a convergence due to the presense of a harmonic vector field, according to Theorem \ref{THM:ErrorEstimate4uh}. The true solution is obtained by combining the ones used in Examples \ref{subsec:ex1} and \ref{subsec:ex5}:
\[
\boldsymbol{u}(x, y, z) =\nabla\times \left\langle 0, 0,  r^{\gamma}
\sin \big(\alpha \theta\big)\right\rangle+
\beta \left(\begin{array}{c}
\sin(\pi x)\cos(\pi y)
\\
-\sin(\pi y)\cos(\pi x)
\\
0
\end{array}
\right).
\]
In this test, we choose $\alpha=2$, $\gamma=2/3$, such that $\boldsymbol{u}\in\boldsymbol{H}^{2/3-\epsilon}$, while the extra term with coefficient $\beta$ is smooth thus not affecting the regularity of the solution. Observe that the extra $\beta$ term is not divergence free. The optimal rate of convergence should be of order $O(h^{2/3})$ on simply-connected domains unaffected by the $\beta$ term. As $\beta$ varies, the numerical results do not demonstrate any convergence for the vector field $\boldsymbol{u}$, while optimal order of convergence are seen for both $(e_{\lambda}, \mathbf{e}_{\boldsymbol{q}})$ and $e_{s}$. The numerical performance is in consistency with our theory as established in Theorem \ref{THM:ErrorEstimate4Triple} for the convergence of $(e_{\lambda}, \mathbf{e}_{\boldsymbol{q}})$ and $e_{s}$ and Theorem
\ref{THM:ErrorEstimate4uh} for the convergence of the vector field $\boldsymbol{u}_h$ up to a harmonic field. The vector fields of $\mathcalQ_h\boldsymbol{u}$ and $\boldsymbol{u}_h$ are plotted in Figure \ref{fig:ex7} (see (a) and (b)), while their difference $\boldeta_h=\mathcalQ_h\boldsymbol{u}-\boldsymbol{u}_h$ is plotted in the same figure as (c). According to Theorem \ref{THM:ErrorEstimate4uh}, the vector field $\boldeta_h$ is an approximate harmonic field with normal boundary condition.

\begin{table}[htb]\caption{Errors and the rates of convergence with different
$\beta$ for Example 7}
\label{table:ex5}
\centering
\begin{tabular}{| c |c |c | c | c | c | c | c |c | c|}
        \hline
 & $1/h$ & $\|\mathbf{e}_{\boldsymbol{u}}\|$  & rate
         &  $\|\varepsilon^{1/2} \mathbf{e}_{\mathcalQ_h\boldsymbol{u}}\|$ & rate
         &$\tnorm{(e_{\lambda}, \mathbf{e}_{\boldsymbol{q}})}$  & rate
          & $\tnorm{e_{s}}$ & rate
\\
       \hline
             &2& 1.26e0 &  -- & 8.08e-1 &  -- & 2.52e0 &  -- & 3.68e-1 &  --
 \\
$\gamma=2/3$ &4& 8.78e-1 & 0.52 & 6.35e-1 & 0.35 & 1.63e0 & 0.63 & 2.54e-1 & 0.53
\\
$\beta=1$    &8& 6.71e-1 & 0.39 & 5.55e-1 & 0.19 & 1.03e0 & 0.66 & 1.58e-1 & 0.68
\\
            &16& 5.82e-1 & 0.21 & 5.33e-1 & 0.06 & 6.44e-1 & 0.67 & 9.72e-2 & 0.70
\\
\hline
 & $1/h$ & $\|\mathbf{e}_{\boldsymbol{u}}\|$  & rate
         &  $\|\varepsilon^{1/2} \mathbf{e}_{\mathcalQ_h\boldsymbol{u}}\|$ & rate
         &$\tnorm{(e_{\lambda}, \mathbf{e}_{\boldsymbol{q}})}$  & rate
          & $\tnorm{e_{s}}$ & rate
\\
 \hline
            &2& 3.42e0 & --     & 2.75e0 &  --  & 5.32e0 &  --  & 6.52e-1 &  --
 \\
$\gamma=2/3$  &4& 2.89e0 & 0.25 & 2.66e0 & 0.05 & 3.16e0 & 0.75 & 4.05e-1 & 0.69
\\
$\beta=5$    &8& 2.71e0 & 0.09 & 2.64e0 & 0.01 & 1.79e0 & 0.82 & 2.21e-1 & 0.87
\\
           &16& 2.66e0 & 0.03 & 2.64e0 & 0.00 & 1.01e0 & 0.82 & 1.27e-1 & 0.80
  \\
 \hline
 \end{tabular}
\end{table}

\begin{figure}[h]
  \centering
\begin{subfigure}[b]{0.4\linewidth}
    \centering
    \includegraphics[width=0.9\textwidth]{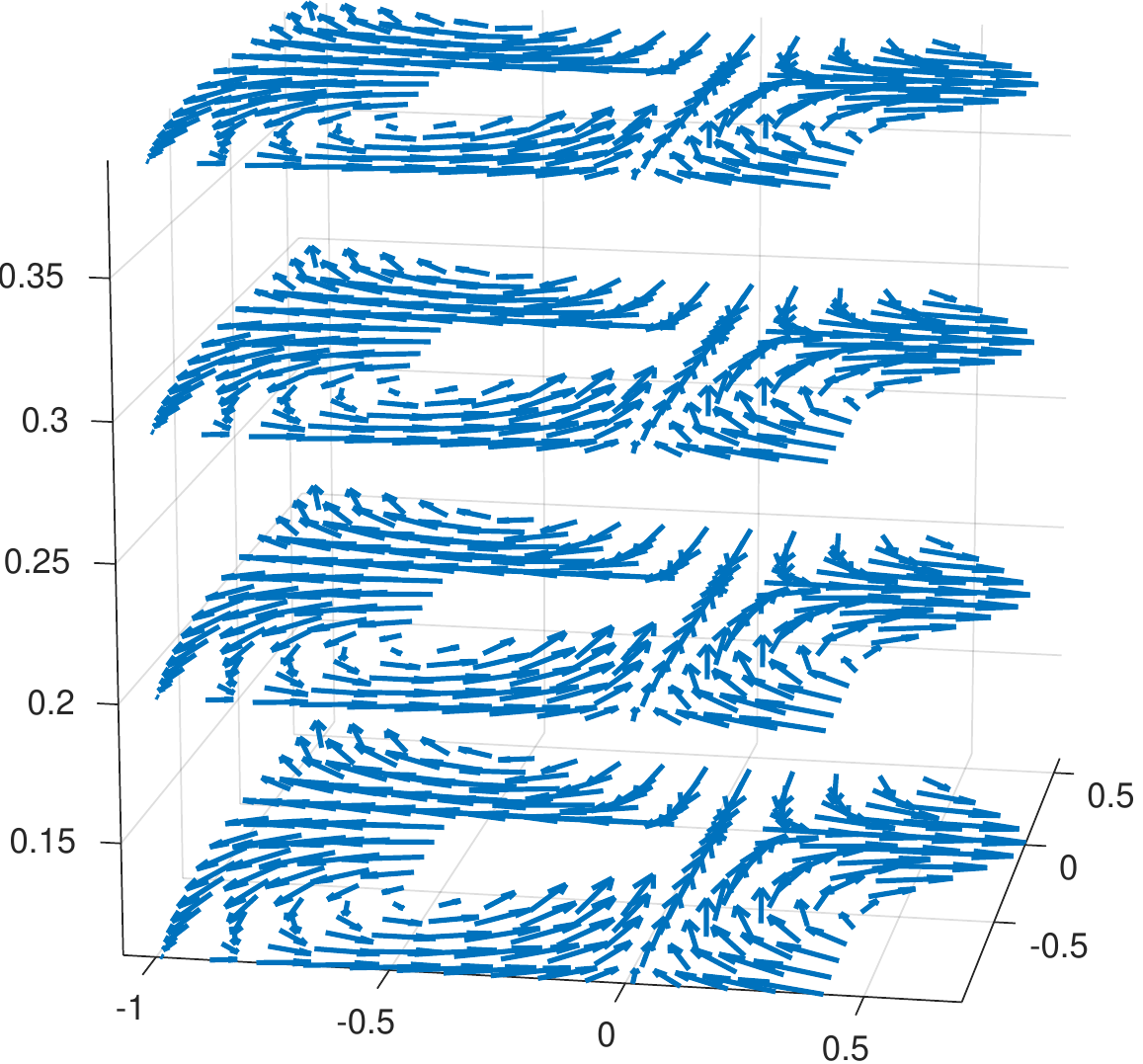}
    \caption{\label{fig:ex7-Qu}}
\end{subfigure}%
\quad
\begin{subfigure}[b]{0.4\linewidth}
      \centering
      \includegraphics[width=0.9\textwidth]{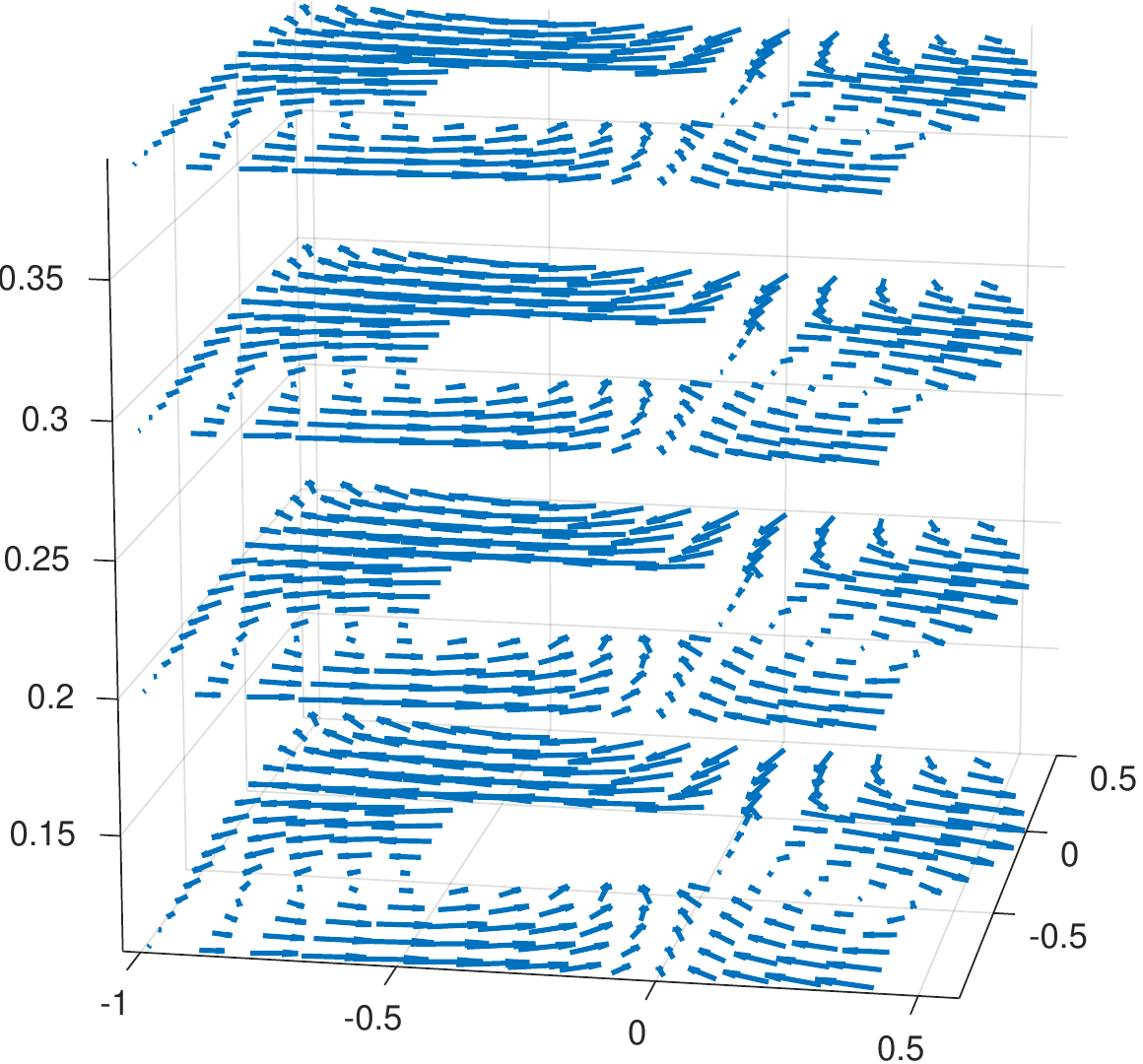}
      \caption{\label{fig:ex7-uwg}}
\end{subfigure}
\quad
\begin{subfigure}[b]{0.4\linewidth}
      \centering
      \includegraphics[width=0.9\textwidth]{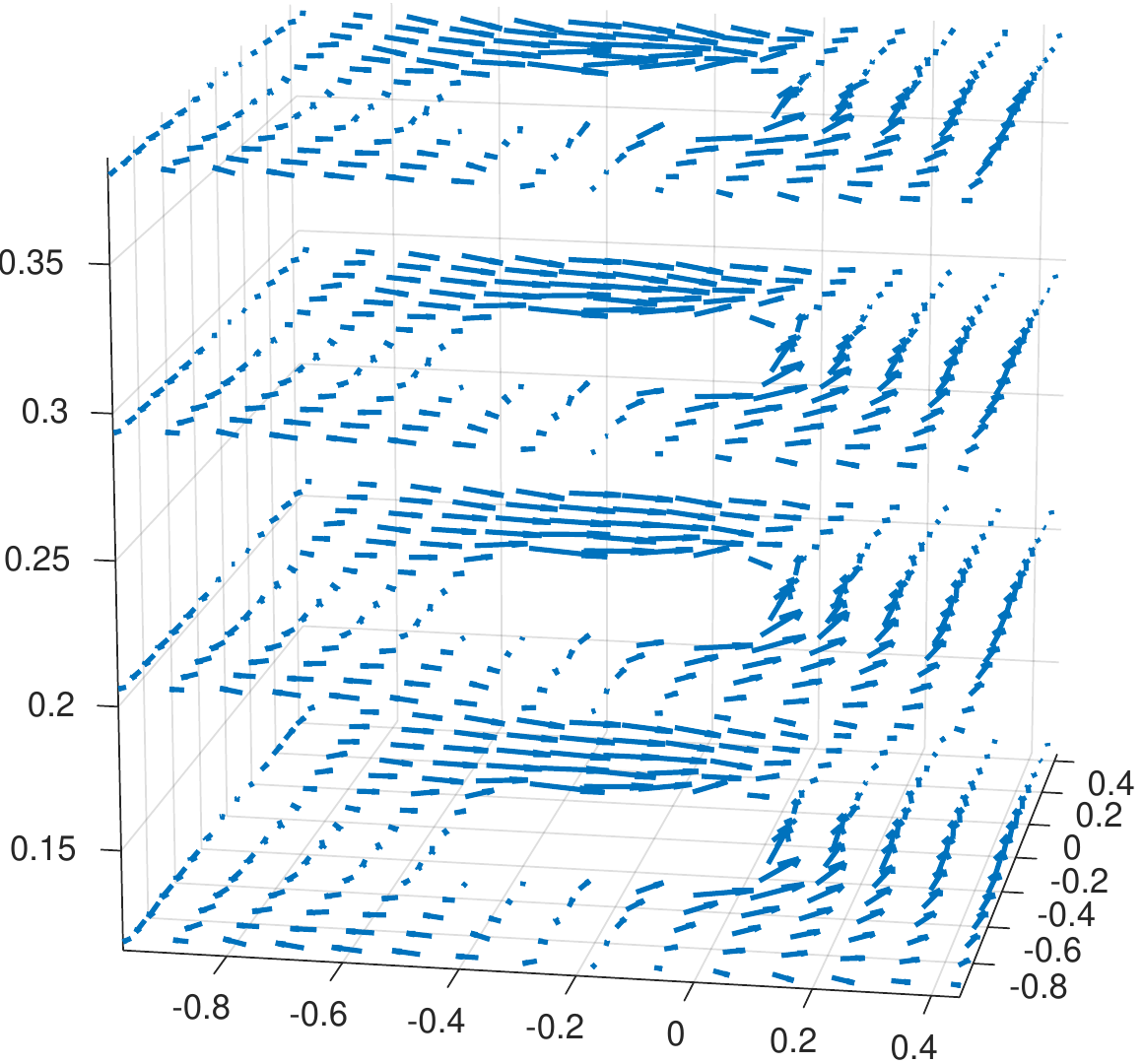}
      \caption{\label{fig:ex7-harmonic}}
\end{subfigure}
\caption{Example \ref{subsec:ex7}: The vector fields of $\mathcalQ_h\boldsymbol{u}$ in
(\subref{fig:ex7-Qu}) and that of the PDWG approximation (\subref{fig:ex7-uwg}) are visually different when $\boldsymbol{u}$ is not divergence-free on a toroidal domain. 
The discrete harmonic field $\boldeta_h=\mathcalQ_h\boldsymbol{u}-\boldsymbol{u}_h$ is plotted in (\subref{fig:ex7-harmonic}). The plots are on several $z=c$ planes. 
}
\label{fig:ex7}
\end{figure}

\newpage
\appendix

\section{Helmholtz Decomposition}
\begin{theorem}\label{THM:helmholtz-2}
For any vector-valued function $\bu\in [L^2(\Omega)]^3$, there exists a unique $\bpsi\in H_0(curl;\Omega),\
\phi\in H^1(\Omega)/\mathbb{R}$, and $\boldeta\in
\mathbb{H}_{\varepsilon n,0}(\Omega)$ such that
\begin{eqnarray}\label{EQ:helmholtz-2}
&&\bu =\varepsilon^{-1}\nabla\times\bpsi + \nabla\phi + {\boldsymbol\eta},\\
&&\nabla\cdot(\varepsilon\bpsi)=0,\ \langle
\varepsilon\bpsi\cdot\bn_i, 1\rangle_{\Gamma_i} = 0, \ i=1,\ldots, L.\label{EQ:helmholtz-2.2}
\end{eqnarray}
Moreover, the following estimate holds true
\begin{equation}\label{EQ:helmholtz-288}
\|\bpsi\|_{H({\rm curl}; \Omega)} + \|\nabla\phi\|_0 \leqC
(\varepsilon\bu,\bu)^{\frac12}.
\end{equation}
\end{theorem}

\begin{proof} The following is a sketch of the proof.
Consider the problem of seeking $\bpsi\in \Wspace$ such that
\begin{equation}\label{EQ:May20-100}
(\varepsilon^{-1}\nabla\times\bpsi, \nabla\times\bvarphi)=(\bu,
\nabla\times\bvarphi), \quad \forall\ \bvarphi\in \Wspace.
\end{equation}
Denote by
$$
a(\bpsi,\bvarphi):=(\varepsilon^{-1}\nabla\times\bpsi, \nabla\times\bvarphi)
$$
the bilinear form defined on $\Wspace$.
We claim that $a(\cdot,\cdot)$ is coercive with respect to the
$H({\rm curl}; \Omega)$-norm. To this end, it suffices to derive the
following estimate
\begin{equation}\label{EQ:May20-101}
\|\bv\|_0 \lesssim \|\nabla\times\bv\|_0, \qquad \bv\in\Wspace.
\end{equation}
In fact, for $\bv\in \Wspace$, from
Theorem 3.4 (Chapter 1) of \cite{girault-raviart}, there exists a vector
potential function $\bomega\in [H^1(\Omega)]^3$ such that
\begin{equation}\label{ForCM}
\varepsilon\bv = \nabla\times\bomega,\ \nabla\cdot\bomega=0,\
\|\bomega\|_1\leqC (\varepsilon\bv,\bv)^{\frac12}.
\end{equation}
Using the integration by parts and the condition $\bv\times\bn=0$ on
$\Gamma$, we have
$$
(\varepsilon\bv,\bv) = (\nabla\times\bomega, \bv)=(\bomega,\nabla\times\bv).
$$
It follows from the Cauchy-Schwarz inequality and (\ref{ForCM}) that
$$
(\varepsilon\bv,\bv) \leq \|\bomega\|_0\ \|\nabla\times\bv\|_0 \lesssim
(\varepsilon\bv,\bv)^{\frac12}\ \|\nabla\times\bv\|_0,
$$
which implies (\ref{EQ:May20-101}).

Now from the Lax-Milgram Theorem, there exists a unique $\bpsi\in
\Wspace$ satisfying the equation (\ref{EQ:May20-100}) such that
$$
\|\bpsi\|_{H({\rm curl}; \Omega)}
\lesssim (\varepsilon
\bu,\bu).
$$
It is easy to see that $\Wspace$ is
equivalent to the following quotient space:
$$
H_0({\rm curl};\Omega)/(\nabla H_{0c}^1(\Omega)) = \{\bv\in H_0({\rm
curl};\Omega):\ (\varepsilon\bv, \nabla\phi)=0,\ \forall \phi\in
H_{0c}^1(\Omega) \}.
$$
Thus, by using a Lagrangian multiplier $p\in H_{0c}^1(\Omega)$, the
problem (\ref{EQ:May20-100}) can be re-formulated as follows: Find
$\bpsi\in H_0({\rm curl};\Omega)$ and $p\in H_{0c}^1(\Omega)$ such
that
\begin{equation}\label{EQ:May20-102}
\begin{split}
(\varepsilon^{-1}\nabla\times\bpsi, \nabla\times\bvarphi)+(\varepsilon\nabla p,
\bvarphi)&=(\bu,
\nabla\times\bvarphi),\quad \forall\ \bvarphi\in H_0({\rm curl};\Omega),\\
(\bpsi, \varepsilon\nabla s)&=0,\quad\qquad\qquad\forall\ s\in
H_{0c}^1(\Omega).
\end{split}
\end{equation}
It follows from the first equation of (\ref{EQ:May20-102}) that
$$
\nabla\times(\bu - \varepsilon^{-1}\nabla\times \bpsi) - \varepsilon\nabla p = 0.
$$
Since $p\in H_{0c}^1(\Omega)$, the two terms on the left-hand
side of the above equation are orthogonal in the $\varepsilon^{-1}$-weighted
$L^2(\Omega)$ norm. Thus, we have
$$
\nabla\times(\bu - \varepsilon^{-1}\nabla\times \bpsi) =0,
$$
which gives
$$
\bu - \varepsilon^{-1}\nabla\times \bpsi \in H^0({\rm curl};\Omega).
$$
Thus, there exist unique $\phi\in H^1(\Omega)/\mathbb{R}$ and
$\boldeta\in \mathbb{H}_{\varepsilon n,0}(\Omega)$ such that
$$
\bu - \varepsilon^{-1}\nabla\times \bpsi =\nabla \phi + \boldeta,
$$
which completes the proof of the theorem.
\end{proof}

\section*{Acknowledgement}
 We would like to express our gratitude to Dr. Long Chen (UCI) for his valuable discussion and suggestions. 


\end{document}